\documentclass[11pt, twoside]{article}
\usepackage{latexsym}
\usepackage{cite}
\usepackage{amsmath}\def\rad{\mbox{rad}}
\usepackage{amssymb}\def\End{\mbox{End}}
\usepackage[all]{xy}
\usepackage{amsfonts}
\usepackage{verbatim}
\usepackage{amsthm}
\usepackage{mathrsfs}
\usepackage{epsfig}
\usepackage{xy}
\usepackage{array}
\usepackage{stmaryrd}
\usepackage{graphicx,color}
\usepackage{xcolor}
\usepackage[colorlinks=true,linkcolor=blue,citecolor=blue]{hyperref}
\usepackage{tikz}
\usetikzlibrary{arrows,calc}
\usepackage{etex}
\usepackage{mathdots}
\usepackage{float}
\usepackage{graphics}
\usepackage{pdflscape}
\usepackage{extarrows}
\usepackage{anysize,hyperref}
\input xypic
\xyoption{all}
\usepackage{bm}

\usepackage[perpage,symbol]{footmisc}
\usepackage{setspace}
\def\P{\mathcal{P}}
\def\I{\mathcal{I}}

\def\C{\mathscr{C}}
\def\E{\mathbb{E}}

\def\s{\mathfrak{s}}
\def\id{\mathrm{id}}
\def\op{^\mathrm{op}}
\def\Ab{\mathsf{Ab}}
\def\del{\delta}
\def\dr{\ar@{->}[r]}

\def\Im{\mbox{Im}}\def\Ker{\mbox{Ker}}
\def\add{\mbox{add}}
\def\Ext{\mbox{Ext}}
\def\Hom{\mbox{Hom}}

\newcommand{\CC}{{\bf{C}}^{n+2}_{\C}}
\newcommand{\Id}{\operatorname{Id}}
\newcommand{\ov}{\overset}
\newcommand{\lra}{\longrightarrow}
\newcommand{\co}{\colon}
\newcommand{\uas}{^{\ast}}            
\newcommand{\sas}{_{\ast}}
\newcommand{\Xd}{\langle X_{\bullet},\del\rangle}  

\newcommand{\ush}{^\sharp}           
\newcommand{\ssh}{_\sharp}\newcommand{\set}[1]{\left\{#1\right\}}
\SelectTips{cm}{10}

\usepackage{setspace}
\begin{document}
\baselineskip=15pt
\title{\Large{\bf Higher Auslander-Reiten sequences via morphisms\\[1mm] determined by objects\footnotetext{\hspace{-1em} Jian He was supported by the National Natural Science Foundation of China (Grant No. 12171230). Panyue Zhou was supported by the National Natural Science Foundation of China (Grant No. 11901190) and the Scientific Research Fund of Hunan Provincial Education Department (Grant No. 19B239).} }}
\medskip
\author{Jian He, Jing He and Panyue Zhou}

\date{}

\maketitle
\def\blue{\color{blue}}
\def\red{\color{red}}

\newtheorem{theorem}{Theorem}[section]
\newtheorem{lemma}[theorem]{Lemma}
\newtheorem{corollary}[theorem]{Corollary}
\newtheorem{proposition}[theorem]{Proposition}
\newtheorem{conjecture}{Conjecture}
\theoremstyle{definition}
\newtheorem{definition}[theorem]{Definition}
\newtheorem{question}[theorem]{Question}
\newtheorem{remark}[theorem]{Remark}
\newtheorem{remark*}[]{Remark}
\newtheorem{example}[theorem]{Example}
\newtheorem{example*}[]{Example}
\newtheorem{condition}[theorem]{Condition}
\newtheorem{condition*}[]{Condition}
\newtheorem{construction}[theorem]{Construction}
\newtheorem{construction*}[]{Construction}

\newtheorem{assumption}[theorem]{Assumption}
\newtheorem{assumption*}[]{Assumption}

\baselineskip=17pt
\parindent=0.5cm

\begin{abstract}
\begin{spacing}{1.2}
Let $(\C,\E,\s)$ be an $\Ext$-finite, Krull-Schmidt and $k$-linear $n$-exangulated category with $k$ a
commutative artinian ring. In this note, we define two additive subcategories $\C_r$ and $ \C_l$ of $\C$ in terms of the representable functors from the stable category of $\C$ to the category of finitely generated $k$-modules. Moreover, we show that there exists an equivalence between the stable categories of these two full subcategories. Finally, we give some equivalent characterizations on the existence of Auslander-Reiten $n$-exangles via determined morphisms. These results unify and extend their works
by Jiao--Le for exact categories,  Zhao--Tan--Huang  for extriangulated categories, Xie--Liu--Yang for $n$-abelian categories.\\[0.2cm]
 \end{spacing}
\end{abstract}

\pagestyle{myheadings}
\markboth{\rightline {\scriptsize J. He, J. He and P. Zhou }}
         {\leftline{\scriptsize Higher Auslander-Reiten sequences via morphisms determined by objects}}

\section{Introduction}
The notion of extriangulated categories was introduced by Nakaoka--Palu in \cite{NP}, which can be viewed as a simultaneous generalization of exact categories and triangulated categories.
 The data of such a category is a triplet $(\C,\E,\s)$, where $\C$ is an additive category, $$\mathbb{E}\colon \C^{\rm op}\times \C \rightarrow \Ab~~\mbox{($\Ab$ is the category of abelian groups)}$$ is an additive bifunctor and $\mathfrak{s}$ assigns to each $\delta\in \mathbb{E}(C,A)$ a class of $3$-term sequences with end terms $A$ and $C$ such that certain axioms hold. Recently, Herschend--Liu--Nakaoka \cite{HLN}
introduced the notion of $n$-exangulated categories for any positive integer $n$. It is not only a higher dimensional analogue of extriangulated categories,
but also gives a common generalization of $n$-exact categories ($n$-abelian categories are also $n$-exact categories) in the sense of
Jasso \cite{Ja} and $(n+2)$-angulated in the sense of Geiss--Keller--Oppermann \cite{GKO}. However, there are some other examples of $n$-exangulated categories which are neither $n$-exact nor $(n+2)$-angulated, see \cite{HLN, LZ,HZZ2}.

Auslander-Reiten theory was introduced by Auslander and Reiten in \cite{AR1,AR2}. Since its introduction, Auslander-Reiten theory has become a fundamental tool for
studying the representation theory of Artin algebras.
Later it has been generalized to these situation of exact categories \cite{Ji}, triangulated categories \cite{H,RV} and its subcategories \cite{AS,J}
and some certain additive categories \cite{L,J,S} by many authors. Iyama, Nakaoka and Palu \cite{INP} developed  Auslander--Reiten theory for extriangulated categories.
This unifies Auslander--Reiten theories in exact categories and triangulated categories independently.

Functors and morphisms determined by objects were introduced by Auslander \cite{Au}.
These concepts provide a method to construct and organise morphisms in additive categories, generalizing previous work of Auslander and Reiten on almost split sequences \cite{AR1}.
Later, Ringel \cite{R2} presented a survey of these results, rearranged them as lattice isomorphisms (the Auslander bijections) and added many examples.
  The concept of a morphism determined by an object provides a method
to construct or classify morphisms in a fixed category.
Krause \cite{K} used the method to construct or classify morphisms in triangulated categories having Serre duality. He \cite{K1} also described a procedure for constructing morphisms in additive categories, combining Auslander's concept of a morphism determined by an object with the existence of flat covers. Suppose $k$ is a commutative artinian ring. Let $\C$ be an $\Ext$-finite, Krull-Schmidt exact category, Jiao--Le \cite{JL} gave an equivalence characterization on the existence of almost split sequences using determined morphisms. Recently, Zhao--Tan--Huang \cite{ZTH} extended Jiao--Le's result to the $\Ext$-finite, Krull-Schmidt $k$-linear extriangulated category. Namely, Zhao--Tan--Huang provided an equivalence characterization on the existence of almost split triangles using determined morphisms.  Subsequently,  Xie-Liu-Yang \cite{XLY} proved a similar result to Jiao--Le and Zhao--Tan--Huang. More precisely,  if $\C$ is a small $\Ext$-finite, Krull-Schmidt $n$-abelian $k$-category with enough projectives and enough injectives, they gave an equivalence characterization on the existence of $n$-Auslander--Reiten sequences via determined morphisms. Based on this idea, we have a natural question of whether their results of Jiao--Le \cite{JL} and Xie-Liu-Yang \cite{XLY} can be extended under the framework of $n$-exangulated categories or whether the result of Zhao--Tan--Huang \cite{ZTH} has a higher counterpart. In this article, we give an affirmative answer. The paper is organized as follows.

In Section 2, we recall the definition of $n$-exangulated category and review some results.

Let $(\C,\E,\s)$ be an $\Ext$-finite, Krull-Schmidt and $k$-linear $n$-exangulated category with $k$ a
commutative artinian ring. In Section 3, we define two additive subcategories of $\C$ as follows:
$$\C_r=\{X\in\C|~\text{the functor}~ D\E(X,-)\colon\overline\C\to  k\text{-mod}~~\text{is representable}\},$$
$$\C_l=\{X\in\C|~\text{the functor}~ D\E(-,X)\colon\underline\C\to k\text{-mod}~~\text{is representable}\},$$
where $D:=\Hom_k(-,\check{k})$ with $\check{k}$ be the minimal injective cogenerator for the category $k$-mod of finitely generated $k$-modules. Meanwhile, based on these two full subcategories, we construct two functors $\tau_n\colon\underline{\C_r}\to\overline{\C_l}\  and~~ \tau_n^-\colon\overline{\C_l}\to\underline{\C_r}\
$, which are equivalences. Moreover, $\tau_n$ and $\tau_n^-$ are quasi-inverse to each other {\rm (see Theorem \ref{qua})}.

In Section 4, we give some equivalent characterizations on the existence of Auslander-Reiten $n$-exangles using determined morphisms under technical conditions
(see Condition \ref{cd}).

\begin{theorem}{\rm (see Theorem \ref{thm} for details)}
Let $C$  be a non-projective indecomposable object in $\C$. The following statements are equivalent:
\begin{enumerate}
\item[$(1)$]
$C\in\C_r$.
\item[$(2)$]\rm
For each $X_{n+1}\in\C$ and each right $\End_{\C}(C)$-submodule $H$ of $\C(C,X_{n+1})$ satisfying $\P(C,X_{n+1})\subseteq H$, there exists a distinguished $n$-exangle
$$X_{\bullet}:X_0\xrightarrow{}X_1\xrightarrow{}X_2\xrightarrow{}\cdots\xrightarrow{}X_n\xrightarrow{\alpha}X_{n+1}\overset{\eta}{\dashrightarrow}$$
where $\alpha$ is right $C$-determined such that $H=\rm\Im~\C(C,\alpha)$.
\item[$(3)$]\rm There exists an inflation $\alpha\colon X_{0}\to X_{1}$ whose intrinsic weak $n$-cokernel is $C$ such that $\alpha$ is left $K$-determined for some object $K$.

\item[$(4)$]\rm There exists an {Auslander-Reiten $n$-exangle} ending at $C$.
\item[$(5)$]\rm
There exists a non-split epimorphism deflation which is right $C$-determined.
\item[$(6)$]\rm
There exists a deflation $\alpha\colon X_{n}\to X_{n+1}$ and a morphism $f\colon C\to X_{n+1}$ such that $f$ almost factors through $\alpha$.

\end{enumerate}
\end{theorem}

\section{Preliminaries}
In this section, we briefly review basic concepts and results concerning $n$-exangulated categories.

{ For any pair of objects $A,C\in\C$, an element $\del\in\E(C,A)$ is called an {\it $\E$-extension} or simply an {\it extension}. We also write such $\del$ as ${}_A\del_C$ when we indicate $A$ and $C$. The zero element ${}_A0_C=0\in\E(C,A)$ is called the {\it split $\E$-extension}. For any pair of $\E$-extensions ${}_A\del_C$ and ${}_{A'}\del{'}_{C'}$, let $\delta\oplus \delta'\in\mathbb{E}(C\oplus C', A\oplus A')$ be the
element corresponding to $(\delta,0,0,{\delta}{'})$ through the natural isomorphism $\mathbb{E}(C\oplus C', A\oplus A')\simeq\mathbb{E}(C, A)\oplus\mathbb{E}(C, A')
\oplus\mathbb{E}(C', A)\oplus\mathbb{E}(C', A')$.

For any $a\in\C(A,A')$ and $c\in\C(C',C)$,  $\E(C,a)(\del)\in\E(C,A')\ \ \text{and}\ \ \E(c,A)(\del)\in\E(C',A)$ are simply denoted by $a_{\ast}\del$ and $c^{\ast}\del$, respectively.

Let ${}_A\del_C$ and ${}_{A'}\del{'}_{C'}$ be any pair of $\E$-extensions. A {\it morphism} $(a,c)\colon\del\to{\delta}{'}$ of extensions is a pair of morphisms $a\in\C(A,A')$ and $c\in\C(C,C')$ in $\C$, satisfying the equality
$a_{\ast}\del=c^{\ast}{\delta}{'}$.}

\begin{definition}\cite[Definition 2.7]{HLN}
Let $\bf{C}_{\C}$ be the category of complexes in $\C$. As its full subcategory, define $\CC$ to be the category of complexes in $\C$ whose components are zero in the degrees outside of $\{0,1,\ldots,n+1\}$. Namely, an object in $\CC$ is a complex $X_{\bullet}=\{X_i,d^X_i\}$ of the form
\[ X_0\xrightarrow{d^X_0}X_1\xrightarrow{d^X_1}\cdots\xrightarrow{d^X_{n-1}}X_n\xrightarrow{d^X_n}X_{n+1}. \]
We write a morphism $f_{\bullet}\co X_{\bullet}\to Y_{\bullet}$ simply $f_{\bullet}=(f_0,f_1,\ldots,f_{n+1})$, only indicating the terms of degrees $0,\ldots,n+1$.
\end{definition}

\begin{definition}\cite[Definition 2.11]{HLN}
By Yoneda lemma, any extension $\del\in\E(C,A)$ induces natural transformations
\[ \del\ssh\colon\C(-,C)\Rightarrow\E(-,A)\ \ \text{and}\ \ \del\ush\colon\C(A,-)\Rightarrow\E(C,-). \]
For any $X\in\C$, these $(\del\ssh)_X$ and $\del\ush_X$ are given as follows.
\begin{enumerate}
\item[\rm(1)] $(\del\ssh)_X\colon\C(X,C)\to\E(X,A)\ :\ f\mapsto f\uas\del$.
\item[\rm (2)] $\del\ush_X\colon\C(A,X)\to\E(C,X)\ :\ g\mapsto g\sas\delta$.
\end{enumerate}
We simply denote $(\del\ssh)_X(f)$ and $\del\ush_X(g)$ by $\del\ssh(f)$ and $\del\ush(g)$, respectively.
\end{definition}

\begin{definition}\cite[Definition 2.9]{HLN}
 Let $\C,\E,n$ be as before. Define a category $\AE:=\AE^{n+2}_{(\C,\E)}$ as follows.
\begin{enumerate}
\item[\rm(1)]  A pair $\Xd$ is an object of the category $\AE$ with $X_{\bullet}\in\CC$
and $\del\in\E(X_{n+1},X_0)$, called an $\E$-attached
complex of length $n+2$, if it satisfies
$$(d_0^X)_{\ast}\del=0~~\textrm{and}~~(d^X_n)^{\ast}\del=0.$$
We also denote it by
$$X_0\xrightarrow{d_0^X}X_1\xrightarrow{d_1^X}\cdots\xrightarrow{d_{n-2}^X}X_{n-1}
\xrightarrow{d_{n-1}^X}X_n\xrightarrow{d_n^X}X_{n+1}\overset{\delta}{\dashrightarrow}.$$
\item[\rm (2)]  For such pairs $\Xd$ and $\langle Y_{\bullet},\rho\rangle$,  $f_{\bullet}\colon\Xd\to\langle Y_{\bullet},\rho\rangle$ is
defined to be a morphism in $\AE$ if it satisfies $(f_0)_{\ast}\del=(f_{n+1})^{\ast}\rho$.

\end{enumerate}
\end{definition}

\begin{definition}\cite[Definition 2.13]{HLN}\label{def1}
 An {\it $n$-exangle} is an object $\Xd$ in $\AE$ that satisfies the listed conditions.
\begin{enumerate}
\item[\rm (1)] The following sequence of functors $\C\op\to\Ab$ is exact.
$$
\C(-,X_0)\xrightarrow{\C(-,\ d^X_0)}\cdots\xrightarrow{\C(-,\ d^X_n)}\C(-,X_{n+1})\xrightarrow{~\del\ssh~}\E(-,X_0)
$$
\item[\rm (2)] The following sequence of functors $\C\to\Ab$ is exact.
$$
\C(X_{n+1},-)\xrightarrow{\C(d^X_n,\ -)}\cdots\xrightarrow{\C(d^X_0,\ -)}\C(X_0,-)\xrightarrow{~\del\ush~}\E(X_{n+1},-)
$$
\end{enumerate}
In particular any $n$-exangle is an object in $\AE$.
A {\it morphism of $n$-exangles} simply means a morphism in $\AE$. Thus $n$-exangles form a full subcategory of $\AE$.
\end{definition}

\begin{definition}\cite[Definition 2.22]{HLN}
Let $\s$ be a correspondence which associates a homotopic equivalence class $\s(\del)=[{}_A{X_{\bullet}}_C]$ to each extension $\del={}_A\del_C$. Such $\s$ is called a {\it realization} of $\E$ if it satisfies the following condition for any $\s(\del)=[X_{\bullet}]$ and any $\s(\rho)=[Y_{\bullet}]$.
\begin{itemize}
\item[{\rm (R0)}] For any morphism of extensions $(a,c)\co\del\to\rho$, there exists a morphism $f_{\bullet}\in\CC(X_{\bullet},Y_{\bullet})$ of the form $f_{\bullet}=(a,f_1,\ldots,f_n,c)$. Such $f_{\bullet}$ is called a {\it lift} of $(a,c)$.
\end{itemize}
In such a case, we simple say that \lq\lq$X_{\bullet}$ realizes $\del$" whenever they satisfy $\s(\del)=[X_{\bullet}]$.

Moreover, a realization $\s$ of $\E$ is said to be {\it exact} if it satisfies the following conditions.
\begin{itemize}
\item[{\rm (R1)}] For any $\s(\del)=[X_{\bullet}]$, the pair $\Xd$ is an $n$-exangle.
\item[{\rm (R2)}] For any $A\in\C$, the zero element ${}_A0_0=0\in\E(0,A)$ satisfies
\[ \s({}_A0_0)=[A\ov{\id_A}{\lra}A\to0\to\cdots\to0\to0]. \]
Dually, $\s({}_00_A)=[0\to0\to\cdots\to0\to A\ov{\id_A}{\lra}A]$ holds for any $A\in\C$.
\end{itemize}
Note that the above condition {\rm (R1)} does not depend on representatives of the class $[X_{\bullet}]$.
\end{definition}

\begin{definition}\cite[Definition 2.23]{HLN}
Let $\s$ be an exact realization of $\E$.
\begin{enumerate}
\item[\rm (1)] An $n$-exangle $\Xd$ is called an $\s$-{\it distinguished} $n$-exangle if it satisfies $\s(\del)=[X_{\bullet}]$. We often simply say {\it distinguished $n$-exangle} when $\s$ is clear from the context.
\item[\rm (2)]  An object $X_{\bullet}\in\CC$ is called an {\it $\s$-conflation} or simply a {\it conflation} if it realizes some extension $\del\in\E(X_{n+1},X_0)$.
\item[\rm (3)]  A morphism $f$ in $\C$ is called an {\it $\s$-inflation} or simply an {\it inflation} if it admits some conflation $X_{\bullet}\in\CC$ satisfying $d_0^X=f$.
\item[\rm (4)]  A morphism $g$ in $\C$ is called an {\it $\s$-deflation} or simply a {\it deflation} if it admits some conflation $X_{\bullet}\in\CC$ satisfying $d_n^X=g$.
\end{enumerate}
\end{definition}

\begin{definition}\cite[Definition 2.27]{HLN}
For a morphism $f_{\bullet}\in\CC(X_{\bullet},Y_{\bullet})$ satisfying $f_0=\id_A$ for some $A=X_0=Y_0$, its {\it mapping cone} $M_{_{\bullet}}^f\in\CC$ is defined to be the complex
\[ X_1\xrightarrow{d^{M_f}_0}X_2\oplus Y_1\xrightarrow{d^{M_f}_1}X_3\oplus Y_2\xrightarrow{d^{M_f}_2}\cdots\xrightarrow{d^{M_f}_{n-1}}X_{n+1}\oplus Y_n\xrightarrow{d^{M_f}_n}Y_{n+1} \]
where $d^{M_f}_0=\begin{bmatrix}-d^X_1\\ f_1\end{bmatrix},$
$d^{M_f}_i=\begin{bmatrix}-d^X_{i+1}&0\\ f_{i+1}&d^Y_i\end{bmatrix}\ (1\le i\le n-1),$
$d^{M_f}_n=\begin{bmatrix}f_{n+1}&d^Y_n\end{bmatrix}$.

{\it The mapping cocone} is defined dually, for morphisms $h_{\bullet}$ in $\CC$ satisfying $h_{n+1}=\id$.
\end{definition}

\begin{definition}\cite[Definition 2.32]{HLN}
An {\it $n$-exangulated category} is a triplet $(\C,\E,\s)$ of additive category $\C$, additive bifunctor $\E\co\C\op\times\C\to\Ab$, and its exact realization $\s$, satisfying the following conditions.
\begin{itemize}
\item[{\rm (EA1)}] Let $A\ov{f}{\lra}B\ov{g}{\lra}C$ be any sequence of morphisms in $\C$. If both $f$ and $g$ are inflations, then so is $g\circ f$. Dually, if $f$ and $g$ are deflations, then so is $g\circ f$.

\item[{\rm (EA2)}] For $\rho\in\E(D,A)$ and $c\in\C(C,D)$, let ${}_A\langle X_{\bullet},c\uas\rho\rangle_C$ and ${}_A\langle Y_{\bullet},\rho\rangle_D$ be distinguished $n$-exangles. Then $(\id_A,c)$ has a {\it good lift} $f_{\bullet}$, in the sense that its mapping cone gives a distinguished $n$-exangle $\langle M^f_{_{\bullet}},(d^X_0)\sas\rho\rangle$.
 \item[{\rm (EA2$\op$)}] Dual of {\rm (EA2)}.
\end{itemize}
Note that the case $n=1$, a triplet $(\C,\E,\s)$ is a  $1$-exangulated category if and only if it is an extriangulated category, see \cite[Proposition 4.3]{HLN}.
\end{definition}

\begin{example}
From \cite[Proposition 4.34]{HLN} and \cite[Proposition 4.5]{HLN},  we know that $n$-exact categories and $(n+2)$-angulated categories are $n$-exangulated categories.
There are some other examples of $n$-exangulated categories
 which are neither $n$-exact nor $(n+2)$-angulated, see \cite{HLN,LZ,HZZ2}.
\end{example}
The following some Lemmas are very useful which are needed later on.

\begin{lemma}\emph{\cite[Lemma 2.12]{LZ}}\label{a1}
Let $(\C,\E,\s)$ be an $n$-exangulated category, and
$$A_0\xrightarrow{\alpha_0}A_1\xrightarrow{\alpha_1}A_2\xrightarrow{\alpha_2}\cdots\xrightarrow{\alpha_{n-2}}A_{n-1}
\xrightarrow{\alpha_{n-1}}A_n\xrightarrow{\alpha_n}A_{n+1}\overset{\delta}{\dashrightarrow}$$
be a distinguished $n$-exangle. Then we have the following exact sequences:
$$\C(-, A_0)\xrightarrow{}\C(-, A_1)\xrightarrow{}\cdots\xrightarrow{}
\C(-, A_{n+1})\xrightarrow{}\E(-, A_{0})\xrightarrow{}\E(-, A_{1})\xrightarrow{}\E(-, A_{2});$$
$$\C(A_{n+1},-)\xrightarrow{}\C(A_{n},-)\xrightarrow{}\cdots\xrightarrow{}
\C(A_0,-)\xrightarrow{}\E(A_{n+1},-)\xrightarrow{}\E(A_{n},-)\xrightarrow{}\E(A_{n-1},-).$$
\end{lemma}

\begin{lemma}\emph{\cite[Proposition 3.6]{HLN}}\label{a2}
\rm Let ${}_A\langle X_{\bullet},\delta\rangle_C$ and ${}_B\langle Y_{\bullet},\rho\rangle_D$ be distinguished $n$-exangles. Suppose that we are given a commutative square
$$\xymatrix{
 X_0 \ar[r]^{{d_0^X}} \ar@{}[dr]|{\circlearrowright} \ar[d]_{a} & X_1 \ar[d]^{b}\\
 Y_0  \ar[r]_{d_0^Y} &Y_1
}
$$
in $\C$. Then there is a morphism $f_{\bullet}\colon \langle X_{\bullet},\delta\rangle\to\langle Y_{\bullet},\rho\rangle$ which satisfies $f_0=a$ and $f_1=b$.
\end{lemma}
\begin{lemma}\rm\cite[Lemma 2.11 ]{HHZZ}\label{y1}
Let $\C$ be an $n$-exangulated category , and $$\xymatrix{
X_0\ar[r]^{f_0}\ar@{}[dr] \ar[d]^{a_0} &X_1 \ar[r]^{f_1} \ar@{}[dr]\ar[d]^{a_1}\ar@{-->}[dl]^{h_1} &X_2 \ar[r]^{f_2} \ar@{}[dr]\ar[d]^{a_2}\ar@{-->}[dl]^{h_2}&\cdot\cdot\cdot \ar[r]\ar@{}[dr] &X_n \ar[r]^{f_n} \ar@{}[dr]\ar[d]^{a_n}&X_{n+1} \ar@{}[dr]\ar[d]^{a_{n+1}} \ar@{-->}[dl]^{h_{n+1}}\ar@{-->}[r]^-{\delta} &\\
{Y_0}\ar[r]^{g_0} &{Y_1}\ar[r]^{g_1}&{Y_2} \ar[r]^{g_2} &\cdot\cdot\cdot \ar[r] &{Y _n}\ar[r]^{g_n}  &{Y_{n+1}} \ar@{-->}[r]^-{\eta} &}
$$
any morphism of distinguished $n$-exangles. Then the following are equivalent:
\begin{itemize}
\item[\rm (1)]There is a morphism $h_1\colon X_1\to Y_0$, such that $h_1f_0=a_0$.

\item[\rm (2)]There is a morphism $h_{n+1}\colon X_{n+1}\to Y_n$, such that $g_nh_{n+1}=a_{n+1}$.

\item[\rm (3)] $ (a_0)_{*}{\delta}=(a_{n+1})^{*}{\eta}=0$.

\item[\rm (4)] $a_{\bullet}=(a_0,a_1,\cdot\cdot\cdot,a_{n+1})\colon\Xd\to\langle Y_{\bullet},\eta\rangle$ is null-homotopic.

\end{itemize}
\end{lemma}

\begin{corollary}\rm\cite[Corollary 2.12 ]{HHZZ}\label{y2}
If $a_{\bullet}$ is the identity on $\Xd$ as above,  then the following are equivalent:
\begin{itemize}
\item[\rm (1)] $f_0$ is a split monomorphism (also known as a section).

\item[\rm (2)] $f_n$ is a split epimorphism (also known as a retraction).

\item[\rm (3)] $ {\delta}=0$.

\item[\rm (4)] $a_{\bullet}$ is null-homotopic.

\end{itemize}
If a distinguished $n$-exangle satisfies one of the above equivalent conditions, it is called \emph{split}.
\end{corollary}
\begin{definition}\label{def2}\cite[Definition 3.14 ]{ZW}and \cite[Definition 3.2]{LZ}
Let $(\C,\E,\s)$ be an $n$-exangulated category. An object $P\in\C$ is called \emph{projective} if, for any distinguished $n$-exangle
$$A_0\xrightarrow{\alpha_0}A_1\xrightarrow{\alpha_1}A_2\xrightarrow{\alpha_2}\cdots\xrightarrow{\alpha_{n-2}}A_{n-1}
\xrightarrow{\alpha_{n-1}}A_n\xrightarrow{\alpha_n}A_{n+1}\overset{\delta}{\dashrightarrow}$$
and any morphism $c$ in $\C(P,A_{n+1})$, there exists a morphism $b\in\C(P,A_n)$ satisfying $\alpha_n\circ b=c$.
We denote the full subcategory of projective objects in $\C$ by $\P$.
Dually, the full subcategory of injective objects in $\C$ is denoted by $\I$.
\end{definition}

\begin{lemma}\rm\label{def2}\cite[Lemma 3.4 ]{LZ} Let $(\C,\E,\s)$ be an $n$-exangulated category. Then the following statements are equivalent for an object $P\in\C$.
\begin{itemize}
\item[\rm (1)] $\E(P,A)=0$ for any $A\in\C$.
\item[\rm (2)] $P$ is projective.
\item[\rm (3)] Any distinguished $n$-exangle
$A_0\xrightarrow{\alpha_0}A_1\xrightarrow{\alpha_1}A_2\xrightarrow{\alpha_2}\cdots\xrightarrow{\alpha_{n-2}}A_{n-1}
\xrightarrow{\alpha_{n-1}}A_n\xrightarrow{\alpha_n}P\overset{\delta}{\dashrightarrow}$ splits.
\end{itemize}
\end{lemma}

We denote by ${\rm rad}_{\C}$ the Jacobson radical of $\C$. Namely, ${\rm rad}_{\C}$ is an ideal of $\C$ such that ${\rm rad}_{\C}(A, A)$
coincides with the Jacobson radical of the endomorphism ring ${\rm End}(A)$ for any $A\in\C$.

\begin{definition}\cite[Definition 3.3 ]{HZ} When $n\geq2$, a distinguished $n$-exangle in $\C$ of the form
$$A_{\bullet}:~~A_0\xrightarrow{\alpha_0}A_1\xrightarrow{\alpha_1}A_2\xrightarrow{\alpha_2}\cdots\xrightarrow{\alpha_{n-2}}A_{n-1}
\xrightarrow{\alpha_{n-1}}A_n\xrightarrow{\alpha_n}A_{n+1}\overset{}{\dashrightarrow}$$
is minimal if $\alpha_1,\alpha_2,\cdots,\alpha_{n-1}$ are in $\rad_{\C}$.

\end{definition}
The following lemma shows that a distinguished $n$-exangle in an equivalence class can be chosen in a minimal way in a Krull-Schmidt $n$-exangulated category.

\begin{lemma}\rm\label{ml}\cite[Lemma 3.4 ]{HZ} Let $\C$ be a Krull-Schmidt $n$-exangulated category, $A_0,A_{n+1}\in\C$. Then for every equivalence class associated with $\E$-extension $\del={}_{A_0}\del_{A_{n+1}}$, there exists a representation
$$A_{\bullet}:~~A_0\xrightarrow{\alpha_0}A_1\xrightarrow{\alpha_1}A_2\xrightarrow{\alpha_2}\cdots\xrightarrow{\alpha_{n-2}}A_{n-1}
\xrightarrow{\alpha_{n-1}}A_n\xrightarrow{\alpha_n}A_{n+1}\overset{\delta}{\dashrightarrow}$$
such that $\alpha_1,\alpha_2,\cdots,\alpha_{n-1}$ are in $\rad_{\C}$. Moreover, $A_{\bullet}$ is a direct summand of every other elements in this equivalent class.

\end{lemma}

\section{An equivalence between two stable subcategories}
Let $k$ is a commutative artinian ring. From now on, we always assume $(\C,\E,\s)$ is an $\Ext$-finite, Krull-Schmidt and $k$-linear $n$-exangulated category.

\subsection{Two subcategories  $\C_r$ and $\C_l$}
\begin{definition}\label{De1}
A morphism $f:X \rightarrow Y$ in $\C$ is called \emph{n-projectively trivial} if for each $Z\in\C$, the induced map $\E(f,Z):\E(Y,Z)\rightarrow\E(X,Z)$ is zero. Dually, a morphism $g:X \rightarrow Y$ in $\C$ is called \emph{n-injectively trivial} if for each $Z\in\C$, the induced map $\E(Z,g):\E(Z,X)\rightarrow\E(Z,Y)$ is zero.
\end{definition}

\begin{lemma}\label{Le0}
Let $f:A \rightarrow B$  be a morphism in $\C$. Then the following statements are equivalent.
\begin{itemize}
  \item[$(1)$] $f$ is {n-projectively trivial}.
  \item[$(2)$] $f$ factors through any deflation $g\colon X_n\to B$.
  \item[$(3)$] For any distinguished $n$-exangle $X_{\bullet}:X_0\xrightarrow{\alpha_0}X_1\xrightarrow{\alpha_1}X_2\xrightarrow{\alpha_2}\cdots\xrightarrow{\alpha_{n-1}}X_n\xrightarrow{g}B\overset{\delta}{\dashrightarrow}$,
if there exists a morphism of distinguished $n$-exangles as follows
 \begin{equation}\label{A}
\begin{array}{l}
\xymatrix{
X_{\bullet}^{\prime}:&X_0^{\prime}\ar[r]^{\alpha_0^{\prime}}\ar@{}[dr] \ar@{=}[d]^{} &X_1^{\prime}\ar[r]^{\alpha_1^{\prime}} \ar@{}[dr]\ar[d]^{\varphi_1}&X_2^{\prime} \ar[r]^{\alpha_2^{\prime}} \ar@{}[dr]\ar[d]^{\varphi_2}&\cdot\cdot\cdot \ar[r]^{\alpha_{n-1}^{\prime} }\ar@{}[dr] &X_n^{\prime} \ar[r]^{g^{\prime}} \ar@{}[dr]\ar[d]^{\varphi_n}&A \ar@{}[dr]\ar[d]^{f} \ar@{-->}[r]^-{f\uas\del} &\\
X_{\bullet}:&{X_0}\ar[r]^{\alpha_0} &{X_1}\ar[r]^{\alpha_1}&{X_2} \ar[r]^{\alpha_2} &\cdot\cdot\cdot \ar[r]^{\alpha_{n-1}} &{X_n}\ar[r]^{g}  &{B} \ar@{-->}[r]^-{\delta} &,}
\end{array}
\end{equation}
then the top distinguished $n$-exangle $X_{\bullet}^{\prime}$ is split.
\end{itemize}
\end{lemma}
\begin{proof}
(1) $\Leftrightarrow$ (3) It is clear by the definition of \emph{n-projectively trivial} morphisms.

(3) $\Rightarrow$ (2) Since  $g^{\prime}$ is a split epimorphism by Corollary \ref{y2}, it follows that there exists a morphism $g^{\prime\prime}: A\to X_n^{\prime}$ such that
$g^{\prime}\circ g^{\prime\prime}=\Id_{A}$. Then we have that $g\circ(\varphi_n\circ g^{\prime\prime})=f\circ g^{\prime}\circ g^{\prime\prime}=f$, and hence (2) holds.

(2) $\Rightarrow$ (3) For any distinguished $n$-exangle $X_{\bullet}:X_0\xrightarrow{\alpha_0}X_1\xrightarrow{\alpha_1}X_2\xrightarrow{\alpha_2}\cdots\xrightarrow{\alpha_{n-1}}X_n\xrightarrow{g}B\overset{\delta}{\dashrightarrow}$,
consider the diagram (\ref{A}). By the assumption of (2), $f$
factors through $g$, and so $\alpha_0^{\prime}$ is a split monomorphism by  Lemma \ref{y1}. Thus $f^\star\delta=0$, that is, the top  distinguished $n$-exangle $X_{\bullet}^{\prime}$ is split.
\end{proof}

\begin{definition}\cite[Definition 3.1]{HHZZ}\label{222} Let $\C$ be an $n$-exangulated category. A distinguished $n$-exangle
$$A_0\xrightarrow{\alpha_0}A_1\xrightarrow{\alpha_1}A_2\xrightarrow{\alpha_2}\cdots\xrightarrow{\alpha_{n-2}}A_{n-1}
\xrightarrow{\alpha_{n-1}}A_n\xrightarrow{\alpha_n}A_{n+1}\overset{\delta}{\dashrightarrow}$$
in $\C$ is called an \emph{Auslander-Reiten $n$-exangle }if
$\alpha_0$ is left almost split, $\alpha_n$ is right almost split and
when $n\geq 2$, $\alpha_1,\alpha_2,\cdots,\alpha_{n-1}$ are in $\rad_{\C}$.
\end{definition}

\begin{lemma}\rm\cite[Lemma 3.3]{HHZZ}\label{r1}
Let $\C$ be an $n$-exangulated category and
$$A_{\bullet}:~~A_0\xrightarrow{\alpha_0}A_1\xrightarrow{\alpha_1}A_2\xrightarrow{\alpha_2}\cdots\xrightarrow{\alpha_{n-2}}A_{n-1}
\xrightarrow{\alpha_{n-1}}A_n\xrightarrow{\alpha_n}A_{n+1}\overset{\delta}{\dashrightarrow}$$
be a distinguished $n$-exangle in $\C$. Then the following statements are equivalent:
\begin{itemize}
\item[\rm (1)] $A_{\bullet}$ is an Auslander-Reiten $n$-exangle;
\item[\rm (2)] ${\rm{End}}(A_0)$ is local, if $n\geq 2$, $ \alpha_1,\cdots,\alpha_{n-1}$ are in ${\rm rad}_{\C}$ and $\alpha_n$ is right almost split;
\item[\rm (3)] ${\rm{End}}(A_{n+1})$ is local, if $n\geq 2$, $\alpha_1,\alpha_2,\cdots,\alpha_{n-1}$ are in ${\rm rad}_{\C}$ and $\alpha_0$ is left almost split.
\end{itemize}
\end{lemma}

Let $A$ and $B$ be two objects in $\C$. We denote by ${\P}(A,B)$ the set of $n$-projectively trivial morphisms from $A$ to $B$.
The \emph{stable category} $\underline{\C}$ of $\C$ is defined as follows, the category whose objects are objects of $\C$ and whose morphisms are elements of
${\underline{\C}}(A,B)={\C}(A,B)/\P(A,B)$. Given a morphism $f\colon A\to B$ in $\C$, we denote by $\underline{f}$ the image of $f$ in $\underline{\C}$. Dually, We denote by ${\I}(A,B)$ the set of $n$-injectively trivial morphisms from $A$ to $B$.
The \emph{costable category} $\overline{\C}$ of $\C$ is defined dually. Given a morphism $g\colon A\to B$ in $\C$, we denote by $\overline{g}$ the image of $g$ in $\overline{\C}$.

Let $\check{k}$ be the minimal injective cogenerator for the category $k$-mod of finitely generated $k$-modules. We write $D:=\Hom_k(-,\check{k})$. For any two $k$-modules $M$ and $N$, a $k$-bilinear form $\langle-,-\rangle:M\times N\rightarrow\check{k}$ is called \emph{non-degenerated}, if for any nonzero element $m\in M$, there exists some $n\in N$ such that $\langle m,n\rangle\neq0$, and for any nonzero element $n\in N$, there exists some $m\in M$ such that $\langle m,n\rangle\neq0$. For each $Y\in\C$, we notice that every $k$-linear form $\phi:\E(X_{n+1},X_0)\rightarrow\check{k}$ determines two $k$-bilinear forms as follows:
$$\langle-,-\rangle_\phi^Y:\overline{\C}(Y,X_0)\times\E(X_{n+1},Y)\rightarrow \check{k},~~~(\overline{f},\delta)\mapsto\phi(f_{\ast}\delta)$$
and
$$_\phi^Y\langle-,-\rangle:{\E}(Y,X_0)\times\underline{\C}(X_{n+1},Y)\rightarrow \check{k},~~~(\delta,\underline{g})\mapsto\phi(g^{\ast}\delta).$$

The following two results are essentially contained in \cite[Proposition 3.1]{LZ1} and \cite[Proposition 4.1]{LZ1}. However, it can be
extended to our setting.

\begin{lemma}\label{Le3}Let $X_{\bullet}:X_0\xrightarrow{\alpha_0}X_1\xrightarrow{\alpha_1}X_2\xrightarrow{\alpha_2}\cdots\xrightarrow{\alpha_{n-1}}X_n\xrightarrow{\alpha_{n}}X_{n+1}\overset{\gamma}{\dashrightarrow}$ be an \emph{Auslander-Reiten $n$-exangle} in $\C$, and $\phi:\E(X_{n+1},X_0)\rightarrow\check{k}$ a $k$-linear form such that $\phi(\gamma)\neq0$.
\begin{enumerate}
\item[$(1)$]For each $Y\in\C$, we have a \emph{non-degenerated} $k$-bilinear form
$$\langle-,-\rangle_\phi^Y:\overline{\C}(Y,X_0)\times\E(X_{n+1},Y)\rightarrow \check{k},~~~(\overline{f},\delta)\mapsto\phi(f_{\ast}\delta).$$
Moreover, the above $k$-bilinear form induces a natural isomorphism
$$\phi_{X_{n+1},Y}:\overline{\C}(Y,X_0)\rightarrow D\E(X_{n+1},Y),~~~\overline{f}\mapsto\langle\overline{f},-\rangle_\phi^Y$$
functorial in $Y\in\C$ with $\phi=\phi_{X_{n+1},X_0}(\overline{\Id_{X_0}})$.
\item[$(2)$]For each $Y\in\C$, we have a \emph{non-degenerated} $k$-bilinear form
$$_\phi^Y\langle-,-\rangle:{\E}(Y,X_0)\times\underline{\C}(X_{n+1},Y)\rightarrow \check{k},~~~(\delta,\underline{g})\mapsto\phi(g^{\ast}\delta).$$
Moreover, the above $k$-bilinear form induces a natural isomorphism
$$\psi_{Y,X_0}:\underline{\C}(X_{n+1},Y)\rightarrow D\E(Y,X_0),~~~\underline{g}\mapsto_\phi^Y\langle-,{\underline{g}}\rangle$$
functorial in $Y\in\C$ with $\phi=\psi_{X_{n+1},X_0}(\underline{\Id_{X_{n+1}}})$.
\end{enumerate}
\end{lemma}
\begin{proof}
 We only prove that $(1)$, dually one can prove $(2)$. Assume that $\phi(\gamma)\neq0$ and $Y\in\C$. We only show that $\langle-,-\rangle_\phi^Y$ is non-degenerated. Let $$Y_{\bullet}:Y\xrightarrow{\beta_0}Y_1\xrightarrow{\beta_1}Y_2\xrightarrow{\beta_2}\cdots\xrightarrow{\beta_{n-1}}Y_n\xrightarrow{\beta_{n}}X_{n+1}\overset{\delta}{\dashrightarrow}$$ be a non-split distinguished $n$-exangle. Since $X_{\bullet}$ is an Auslander-Reiten $n$-exangle, there exists a morphism $f_n$ such that $\alpha_nf_n=\beta_n$. By the dual of Lemma \ref{a2}, we have the following commutative diagram
$$\xymatrix{
Y_{\bullet}:& Y\ar[r]^{\beta_0}\ar@{}[dr] \ar@{-->}[d]^{f} &Y_1 \ar[r]^{\beta_1} \ar@{}[dr]\ar@{-->}[d]^{}&\cdot\cdot\cdot \ar[r]^{\beta_{n-2}} \ar@{}[dr]&Y_{n-1} \ar[r]^{\beta_{n-1}}\ar@{}[dr]\ar@{-->}[d]^{} &Y_n \ar[r]^{\beta_n} \ar@{}[dr]\ar[d]^{f_n}&X_{n+1} \ar@{}[dr]\ar@{=}[d]^{} \ar@{-->}[r]^-{\delta} &\\
X_{\bullet}:& {X_0}\ar[r]^{\alpha_0} &{X_1}\ar[r]^{\alpha_1}&\cdot\cdot\cdot\ar[r]^{\alpha_{n-2}} &{X_{n-1}}  \ar[r]^{\alpha_{n-1}} &{X _n}\ar[r]^{\alpha_n}  &{X_{n+1}} \ar@{-->}[r]^-{\gamma} &.}
$$
Hence $f_{\ast}\delta=\gamma$ and $f\in\overline{\C}(Y,X_0)$. Then we have that $\langle\overline{f},\delta\rangle_\phi^Y=\phi(f_{\ast}\delta)=\phi(\gamma)\neq0$.

On the other hand, if $0\neq\overline{f}:Y\rightarrow X_0$ is a morphism in $\overline{\C}$, then ${f}:Y\rightarrow X_0$ representing $\overline{f}$ is not $n$-injectively trivial, and there exist $Z\in{\C}$ and $\epsilon\in\E(Z,Y)$ such that $f_{\ast}\epsilon=\E(Z,f)(\epsilon)\neq0$ by the dual of Lemma \ref{Le0}. Which means that there is a commutative diagram of non-split distinguished $n$-exangles
$$\xymatrix{
Z_{\bullet}:& Y\ar[r]^{\eta_0}\ar@{}[dr] \ar[d]^{f} &Z_1 \ar[r]^{\eta_1} \ar@{}[dr]\ar[d]^{}&\cdot\cdot\cdot \ar[r]^{\eta_{n-2}} \ar@{}[dr]&Z_{n-1} \ar[r]^{\eta_{n-1}}\ar@{}[dr]\ar[d]^{} &Z_n \ar[r]^{\eta_n} \ar@{}[dr]\ar[d]^{f_n}&Z \ar@{}[dr]\ar@{=}[d]^{} \ar@{-->}[r]^-{\epsilon} &\\
U_{\bullet}:& {X_0}\ar[r]^{\zeta_0} &{U_1}\ar[r]^{\zeta_1}&\cdot\cdot\cdot\ar[r]^{\zeta_{n-2}} &{U_{n-1}}  \ar[r]^{\zeta_{n-1}} &{U _n}\ar[r]^{\zeta_n}  &{Z} \ar@{-->}[r]^-{f_{\ast}\epsilon} &.}
$$
Since $X_{\bullet}$ is an Auslander-Reiten $n$-exangle, there exists a morphism $h_1$ such that $h_1\alpha_0=\zeta_0$. By Lemma \ref{a2}, we have the following commutative diagram
$$\xymatrix{
X_{\bullet}:& X_0\ar[r]^{\alpha_0}\ar@{}[dr] \ar@{=}[d]^{} &X_1 \ar[r]^{\alpha_1} \ar@{}[dr]\ar[d]^{h_1}&\cdot\cdot\cdot \ar[r]^{\alpha_{n-2}} \ar@{}[dr]&X_{n-1} \ar[r]^{\alpha_{n-1}}\ar@{}[dr]\ar@{-->}[d]^{} &X_n \ar[r]^{\alpha_n} \ar@{}[dr]\ar@{-->}[d]^{}&X_{n+1} \ar@{}[dr]\ar@{-->}[d]^{h} \ar@{-->}[r]^-{\gamma} &\\
U_{\bullet}:& {X_0}\ar[r]^{\zeta_0} &{U_1}\ar[r]^{\zeta_1}&\cdot\cdot\cdot\ar[r]^{\zeta_{n-2}} &{U_{n-1}}  \ar[r]^{\zeta_{n-1}} &{U_n}\ar[r]^{\zeta_n}  &{Z} \ar@{-->}[r]^-{f_{\ast}\epsilon} &.}
$$
Then $\gamma=h^{\ast}(f_{\ast}\epsilon)=f_{\ast}h^{\ast}\epsilon$, therefore, we have that $\langle\overline{f},h^{\ast}\epsilon\rangle_\phi^Y=\phi(f_{\ast}(h^{\ast}\epsilon))=\phi(\gamma)\neq0$.
\end{proof}

\begin{lemma}\label{Le4}Suppose $X_{n+1}$ $(\mbox{respectively},~ Y_{0})$ is a non-projective $(\mbox{respectively, non-injective})$ indecomposable object in $\C$.
\begin{enumerate}
\item[$(1)$]If $\phi_{X_{n+1},-}:\overline{\C}(-,X')\rightarrow D\E(X_{n+1},-)$ is an isomorphism of functors for some $X'\in\C$. Then there exists an \emph{Auslander-Reiten $n$-exangle} ending at $X_{n+1}$ in $\C$ of the form $$X_{\bullet}:X_0\xrightarrow{}X_1\xrightarrow{}X_2\xrightarrow{}\cdots\xrightarrow{}X_n\xrightarrow{}X_{n+1}\overset{}{\dashrightarrow},$$
where $X_0$ is a non-injective indecomposable direct summand of $X'$.

\item[$(2)$]If $\psi_{-,Y_{0}}:\underline{\C}(Y',-)\rightarrow D\E(-,Y_{0})$ is an isomorphism of functors for some $Y'\in\C$. Then there exists an \emph{Auslander-Reiten $n$-exangle} starting at $ Y_{0}$ in $\C$ of the form $$Y_{\bullet}:Y_0\xrightarrow{}Y_1\xrightarrow{}Y_2\xrightarrow{}\cdots\xrightarrow{}Y_n\xrightarrow{}Y_{n+1}\overset{}{\dashrightarrow},$$
where $Y_{n+1}$ is a non-projective indecomposable direct summand of $Y'$.
\end{enumerate}
\end{lemma}
\begin{proof}
 We just prove the first statement, the second statement proves similarly. For each object and each morphism $f:U\rightarrow X'$, we have the following commutative diagram by the naturalness of $ \phi_{X_{n+1},-}$
 $$\xymatrix@C=1.3cm@R=1.2cm{
 \overline{\C}(X',X')\ar[r]^{\phi_{X_{n+1},X'}}_{\simeq} \ar@{}[dr]|{\circlearrowright} \ar[d]_{\overline{\C}(f,X')} &D\E(X_{n+1},X') \ar[d]^{D\E(X_{n+1},f)}\\
\overline{\C}(U,X') \ar[r]_{\phi_{X_{n+1},U}}^{\simeq} &D\E(X_{n+1},U).
}
$$
Take $\phi=\phi_{X_{n+1},X'}(\overline{\Id_{X'}})$, then we have that $$\phi_{X_{n+1},U}(\overline{f})=({D\E(X_{n+1},f)}\circ\phi_{X_{n+1},X'})(\overline{\Id_{X'}})={D\E(X_{n+1},f)}(\phi)=\phi\circ\E(X_{n+1},f).$$ Thus we obtain a $k$-linear form $\phi:\E(X_{n+1},X')\rightarrow\check{k}$ satisfying $\phi_{X_{n+1},U}(\overline{f})(\theta)=\phi(f_{\ast}\theta)$ for each $\theta \in \E(X_{n+1},U)$. Let $X_0$ is a non-injective indecomposable direct summand of $X'$. Then the isomorphism $\phi_{X_{n+1},X_{0}}$ induces a {non-degenerated} $k$-bilinear form
$$\langle-,-\rangle_\phi^{X_0}:\overline{\C}(X_0,X')\times\E(X_{n+1},X_0)\rightarrow \check{k},~~~(\overline{f},\delta)\mapsto\phi(f_{\ast}\delta).$$
Let $$G=\{f\in {\C}(X_0,X')~|~ f ~\mbox{is non-split monomorphism} \}.$$
Since $X_0$ is indecomposable, we know that the nonempty subset $G$ is a $k$-submodule of ${\C}(X_0,X')$. Since $X_0$ is non-injective, we have $\I(X_0,X')\subseteq G$. Hence $\overline{G}:=G/\I(X_0,X')$ is properly contained in ${\C}(X_0,X')$. Then there exists a non-split $\E$-extension $\delta\in\E(X_{n+1},X_{0})$ of the form
$$X_{\bullet}:X_0\xrightarrow{\alpha_0}X_1\xrightarrow{\alpha_1}X_2\xrightarrow{\alpha_2}\cdots\xrightarrow{\alpha_{n-1}}X_n\xrightarrow{\alpha_{n}}X_{n+1}\overset{\delta}{\dashrightarrow}$$
such that $\langle\overline{h},\delta\rangle_\phi^{X_0}=\phi(h_{\ast}\delta)=0$ for each non-split monomorphisms $h:X_0\rightarrow X'$ in $G$. Moreover, we can assume that $\alpha_i\in\rad_{\C}$ for $ i\in\{1,2,\cdots,n-1\}$ by Lemma \ref{ml}.

Finally, we claim that $X_{\bullet}$ an Auslander-Reiten $n$-exangle. Suppose that $s:{X_0}\rightarrow V$ is not a split monomorphism, then for each $t:V\rightarrow X'$, the morphism $t\circ s\in G$. Hence we have $\langle\overline{t\circ s},\delta\rangle_\phi^{X_0}=0$. Consider the {non-degenerated} $k$-bilinear form
$$\langle-,-\rangle_\phi^{V}:\overline{\C}(V,X')\times\E(X_{n+1},V)\rightarrow \check{k},~~~(\overline{f},\beta)\mapsto\phi(f_{\ast}\beta),$$ which is induced by $\phi_{X_{n+1},V}$. Hence we have
$$\langle\overline{t},s_{\ast}\delta\rangle_\phi^{V}=\phi(t_{\ast}(s_{\ast}\delta))=\phi((t\circ s)_{\ast}\delta)=\langle\overline{t\circ s},\delta\rangle_\phi^{X_0}=0.$$
It follows that the $\E$-extension $s_{\ast}\delta$ split by the non-degeneracy of $\langle-,-\rangle_\phi^{V}$. which implies that the morphism $s$ factors through $\alpha_0$ by Lemma \ref{y1}. This shows that the morphism $\alpha_0$ is left almost split. Therefore, the claim follows form Lemma \ref{r1} since ${\rm{End}}(X_{n+1})$ is local.
\end{proof}

We introduce two full subcategories of $\C$ as follows.

$$\C_r=\{X\in\C|~\text{the functor}~ D\E(X,-)\colon\overline\C\to  k\text{-mod}~~\text{is representable}\},$$
$$\C_l=\{X\in\C|~\text{the functor}~ D\E(-,X)\colon\underline\C\to  k\text{-mod}~~\text{is representable}\}.$$

As a direct consequence of Lemmas~\ref{Le3} and Lemma \ref{Le4}, we have the following description of indecomposable objects in $\C_r$ and $\C_l$.

\begin{proposition}\label{desc}
Let $X$ be an indecomposable object in $\C$.
\begin{enumerate}
\item[$(1)$] If $X$ is non-projective, then $X\in\C_r$ if and only if there exists an {Auslander-Reiten $n$-exangle} ending at $X$.
\item[$(2)$] If $X$ is non-injective, then $X\in\C_l$ if and only if there exists an {Auslander-Reiten $n$-exangle} starting at $X$.
\end{enumerate}
\end{proposition}
\subsection{Two functors $\tau_n$ and $\tau_n^{-}$}

For each $X$ in $\C_r$, we define $\tau_n X$ to be an object in $\C$ such that there exists an isomorphism of functors
$$\phi_{X,-}:\overline{\C}(-,\tau_n X)\rightarrow D\E(X,-).$$
Then $\tau_n$ gives a map from the objects of $\C_r$ to that of $\C$. Dually, for each $Y$ in $\C_l$,
we define $\tau_n^- Y$ to be an object in $\C$ such that there exists an isomorphism of functors
$$\psi_{-,Y}:\underline{\C}(\tau_n^- Y,-)\rightarrow D\E(-,Y).$$
Then $\tau_n^- $ gives a map from the objects of $\C_l$ to that of $\C$.

We have the following easy observation.

\begin{lemma}\label{lem7}
Let $X$ and $Y$ be two objects in $\C$.
\begin{enumerate}
\item[$(1)$] If $X\in\C_r$ and $X\cong Y$ in $\underline\C$, then $Y\in\C_r$ and $\tau_n X\cong\tau_n Y$ in $\overline\C$.
\item[$(2)$] If $X\in\C_l$ and $X\cong Y$ in $\overline\C$, then $Y\in\C_l$ and $\tau_n^{-} X\cong\tau_n^{-} Y$ in $\underline\C$.
\end{enumerate}
\end{lemma}
\begin{proof} We only prove that $(1)$, dually one can prove $(2)$. The first assertion follows from the fact that $D\E(X,-)\cong D\E(Y,-)$ as functors, and then we deduce that $\overline{\C}(-,\tau_n X)\cong \overline{\C}(-,\tau_n Y)$. Hence the second assertion follows from  Yoneda's lemma.
\end{proof}

\begin{proposition}\label{lem11}
\begin{enumerate}
\item[]
\item[$(1)$] If $X_{n+1}\in\C_r$, then $\tau_nX_{n+1}\in\C_l$ and $X_{n+1}\cong\tau_n^-\tau_n X_{n+1}$ in $\underline{\C}$.
\item[$(2)$] If $Y_0\in\C_l$, then $\tau_n^{-}Y_0\in\C_r$ and $Y_0\cong\tau_n\tau_n^- Y_0$ in $\overline{\C}$.

\end{enumerate}
\end{proposition}
\begin{proof} We just prove the first statement, the second statement proves similarly. Without loss of generality, we assume that $X_{n+1}$ is indecomposable and non-projective in $\C_r$. By Proposition \ref{desc}, there exists an Auslander-Reiten $n$-exangle $$X_{\bullet}:X_0\xrightarrow{}X_1\xrightarrow{}X_2\xrightarrow{}\cdots\xrightarrow{}X_n\xrightarrow{}X_{n+1}\overset{\gamma}{\dashrightarrow}.$$
Then we have $\overline{\C}(-,X_0)\cong D\E(X_{n+1},-)$ and $\underline{\C}(X_{n+1},-)\cong D\E(-,X_0)$ by Lemma \ref{Le3}.
We then obtain $X_0\in\C_l$. It follows from the Yoneda's lemma that $\tau_nX_{n+1}\cong X_0$ in $\overline\C$, and $\tau_n^-X_0\cong X_{n+1}$
in $\underline\C$. So $\tau_nX_{n+1}\in\C_l$ by Lemma \ref{lem7}(2), and hence
$\tau_n^-\tau_n X_{n+1}\cong \tau_n^-X_0\cong X_{n+1}$ in $\underline\C$.
\end{proof}

Let $\underline{\C_r}$ be the image of ${\C_r}$ under the canonical functor $\C\rightarrow \underline{\C}$ and $\underline{\C_l}$ be the image of ${\C_l}$ under the canonical functor $\C\rightarrow \overline{\C}$. For each morphism $f\colon X_{n+1}\to M_{n+1}$ in $\C_r$, we define the morphism $\tau_n(f)\colon\tau_n X_{n+1}\to \tau_n M_{n+1}$ in $\overline{\C_l}$
such that the following diagram commutes
\[\xymatrix{
  \overline\C(-,\tau_n X_{n+1})\ar[r]^-{\phi_{X_{n+1},-}}\ar[d]_-{\overline\C(-,\tau_n(f))}
    &D\E(X_{n+1} ,-)\ar[d]^-{D\E(f,-)}\\
  \overline\C(-,\tau_n M_{n+1})\ar[r]^-{\phi_{M_{n+1},-}}
    &D\E(M_{n+1},-).
}\]
By the Yoneda's lemma, we know that the morphism $\tau_n(f)$ is uniquely determined from the above commutative diagram. Hence it follows that $\tau_n$ is a functor from $\C_r$ to $\overline{\C_l}$. Moreover, if $f$ is n-projectively trivial, then $D\E(f,-)=0$ and thus $\tau_n(f)=0$ in $\overline{\C_l}$. Thus $\tau_n$ induces a functor from
$\underline{\C_r}$ to $\overline{\C_l}$ which we still denote by $\tau_n$.

Similarly, we have a functor $\tau_n^-\colon\overline{\C_l}\to\underline{\C_r}$. For each $\overline g\colon Y_0\to N_0$ in $\overline{\C_l}$,
the morphism $\tau_n^-(\overline g)\colon \tau_n^-Y_0\to\tau_n^-N_0$ is given by the following commutative diagram
\[\xymatrix{
  \underline\C(\tau_n^- N_0,-)\ar[r]^-{\psi_{-,N_0}}\ar[d]_-{\underline\C(\tau_n^-(\overline g),-)}
    &D\E(- ,N_0)\ar[d]^-{D\E(-,g)}\\
  \underline\C(\tau_n^- Y_0,-)\ar[r]^-{\psi_{-,Y_0}}
    &D\E(- ,Y_0).
}\]

For each $X\in\underline{\C_r}$, we set
\[\underline{\vartheta_X}:=\psi_{X,\tau_nX}^{-1}(\phi_{X, \tau_nX}(\overline{\Id_{\tau_nX}}))\colon \tau_n^-\tau_n X\longrightarrow X\]
in $\underline{\C_r}$. Dually, for each $Y\in\overline{\C_l}$, set
\[\overline{\xi_Y}:=\phi_{\tau_n^- Y,Y}^{-1} (\psi_{\tau_n^- Y,Y}(\underline{\Id_{\tau_n^- Y}}))\colon Y\longrightarrow \tau_n\tau_n^- Y\]
in $\overline{\C_l}$.

\begin{theorem}\label{qua}
The two functors $\tau_n\colon\underline{\C_r}\to\overline{\C_l}\  and~~ \tau_n^-\colon\overline{\C_l}\to\underline{\C_r}\
$ are equivalences. Moreover, $\tau_n$ and $\tau_n^-$ are quasi-inverse to each other.

\end{theorem}
\begin{proof}
We only prove that $\underline\vartheta: \tau_n^-\tau_n\rightarrow \mbox{Id}_{\underline{\C_r}}$
is a natural isomorphism, the other is similar.

\textbf{Step 1:} We prove that $\underline\vartheta$ is a natural transformation.
Indeed, for each $\underline f\colon X_{n+1}\to U_{n+1}$ in $\underline{\C_r}$, consider the following diagram
  \[\xymatrix@C=1.6cm@+1.5em{
    \overline\C(\tau_n  X_{n+1},\tau_n  X_{n+1})\ar[r]^-{\phi_{ X_{n+1},\tau_n  X_{n+1}}}\ar@{}[dr]|{\clubsuit}\ar[d]|{\overline\C(\tau_n  X_{n+1},\tau_n(\underline f))}
      &D\E(X_{n+1},\tau _n  X_{n+1})\ar[d]|{D\E(f,\tau _n  X_{n+1})}\ar@{}[dr]|{\spadesuit}
      &\underline\C(\tau_n^-\tau_n X_{n+1},X_{n+1})\ar[l]_-{\psi_{X_{n+1},\tau_n  X_{n+1}}}\ar[d]|{\underline\C(\tau_n^-\tau_n X_{n+1},\underline f)}\\
    \overline\C(\tau _n  X_{n+1},\tau _n  U_{n+1})\ar[r]^-{\phi_{U_{n+1},\tau _n  X_{n+1}}}
      &D\E(U_{n+1},\tau _n  X_{n+1})
      &\underline\C(\tau_n^-\tau_n X_{n+1},U_{n+1})\ar[l]_-{\psi_{U_{n+1},\tau_n X_{n+1}}}.
  }\]
The square $\clubsuit$ commutes by the definition of $\tau_n(\underline f)$, and the square $\spadesuit$ commutes since the isomorphism
$\psi_{-,\tau _n X_{n+1}}$ is natural. By a diagram chasing, we have
\[\tau _n(\underline f)=\phi_{U_{n+1},\tau_n X_{n+1}}^{-1}(\psi_{U_{n+1},\tau_n X_{n+1}}(\underline f\circ\underline{\vartheta_{X_{n+1}}})).\]
We also have the following commutative diagram
 \[\xymatrix@C=1.6cm@+1.5em{
    \overline\C(\tau_n  U_{n+1},\tau_n  U_{n+1})\ar[r]^-{\phi_{ U_{n+1},\tau_n  U_{n+1}}}\ar[d]|{\overline\C(\tau_n(\underline f),\tau_n  U_{n+1})}
      &D\E(U_{n+1},\tau _n  U_{n+1})\ar[d]|{D\E(U_{n+1},\tau _n  (\underline f))}\ar@{}[dr]|{\clubsuit}
      &\underline\C(\tau_n^-\tau_n U_{n+1},U_{n+1})\ar[l]_-{\psi_{U_{n+1},\tau_n  U_{n+1}}}\ar[d]|{\underline\C(\tau_n^-\tau_n (\underline f),U_{n+1})}\\
    \overline\C(\tau _n  X_{n+1},\tau _n  U_{n+1})\ar[r]^-{\phi_{U_{n+1},\tau _n  X_{n+1}}}
      &D\E(U_{n+1},\tau _n  X_{n+1})
      &\underline\C(\tau_n^-\tau_n X_{n+1},U_{n+1})\ar[l]_-{\psi_{U_{n+1},\tau_n X_{n+1}}}.
  }\]
The square $\clubsuit$ commutes by the definition of $\tau_n ^-\tau_n (\underline f)$. By a diagram chasing, we have
\[\tau _n(\underline f)=\phi_{U_{n+1},\tau_n X_{n+1}}^{-1}(\psi_{U_{n+1},\tau_n X_{n+1}}(\underline {\vartheta_{U_{n+1}}}\circ\tau_n ^-\tau_n (\underline f))).\]
We then obtain $\underline f\circ\underline{\vartheta_{X_{n+1}}}=\underline {\vartheta_{U_{n+1}}}\circ\tau_n ^-\tau_n (\underline f)$.
It follows that $\underline\vartheta$ is a natural transformation.

\textbf{Step 2:} We prove that $\underline{\vartheta_{X_{n+1}}}$ is an isomorphism for each $X_{n+1}\in\C_r$. Without loss of generality, we assume that $X_{n+1}$ is indecomposable and non-projective in $\C$. Let $\alpha=\psi_{\tau_n^-\tau_n X_{n+1},\tau_n X_{n+1}}(\underline{\Id_{\tau_n^-\tau_n X_{n+1}}})$
in $D\E(\tau_n^-\tau_n X_{n+1},\tau_n X_{n+1})$ and let $\beta=\phi_{X_{n+1},\tau_n X_{n+1}}(\overline{\Id_{\tau_n X_{n+1}}})$ in $D\E(X_{n+1},\tau_n X_{n+1})$.
By the definition of $\underline{\vartheta_{X_{n+1}}}$, we have $\beta = \psi_{X_{n+1},\tau_n X_{n+1}} (\underline{\vartheta_{X_{n+1}}})$.
Consider the following commutative diagram
\[\xymatrix@C+5em{
    \underline\C(\tau_n^-\tau_n X_{n+1},\tau_n^-\tau_n X_{n+1})\ar[r]^-{\psi_{{\tau_n^-\tau_n X_{n+1}},\tau_n X_{n+1}}}\ar[d]_-{\underline\C(\tau_n^-\tau_n X_{n+1},\underline{\vartheta_{X_{n+1}}})}
      &D\E({\tau_n^-\tau_n X_{n+1}},\tau_n X_{n+1})\ar[d]^-{D\E(\vartheta_{X_{n+1}},\tau_n X_{n+1})}\\
    \underline\C(\tau_n^-\tau_n X_{n+1},X_{n+1})\ar[r]^-{\psi_{{ X_{n+1}},\tau_n X_{n+1}}}
      &D\E({ X_{n+1}},\tau_n X_{n+1}).
  }\]
Since $\psi_{\tau_n^-\tau_n X_{n+1},\tau_n X_{n+1}}(\underline{\Id_{\tau_n^-\tau_n X_{n+1}}})=\alpha$ and
$\underline\C(\tau_n^-\tau_n X_{n+1},\underline{\vartheta_{X_{n+1}}})(\underline{\Id_{\tau_n^-\tau_n X_{n+1}}})=\underline{\vartheta_{X_{n+1}}}$, we have
\[\beta = D\E(\vartheta_{X_{n+1}},\tau_n X_{n+1})(\alpha) = \alpha\circ\E(\vartheta_{X_{n+1}},\tau_n X_{n+1}).\]

On the other hand, since $X_{n+1}$ is non-projective in $\C$, $ X_{0}\cong\tau_n X_{n+1}$ in $\overline\C$ is nonzero and then non-injective in $\C$. Thus, there is an isomorphism $\phi_{X_{n+1},-}:\overline{\C}(-,X_{0})\rightarrow D\E(X_{n+1},-)$. By Lemma \ref{Le4}(1), there exists an {Auslander-Reiten $n$-exangle} of the form $$X_{\bullet}:X_0\xrightarrow{}X_1\xrightarrow{}X_2\xrightarrow{}\cdots\xrightarrow{}X_n\xrightarrow{}X_{n+1}\overset{\eta}{\dashrightarrow},$$
By Lemma~\ref{Le3}(1), we have a natural isomorphism $\phi'_{X_{n+1},-}:\overline{\C}(-,X_{0})\rightarrow D\E(X_{n+1},-)$
such that $\phi'_{X_{n+1},X_{0}}(\overline{\Id_{X_{0}}})(\eta)$ $\ne 0$. Setting $\beta':=\phi'_{X_{n+1},X_{0}}(\overline{\Id_{X_{0}}})$,
we have $\beta'(\eta)\neq 0$. By the Yoneda's lemma, there exists some $s\colon X_{0}\to\tau_n X_{n+1}$ such that
$\overline{\C}(-,s) = \phi_{X_{n+1},-}^{-1}\circ\phi'_{X_{n+1},-}$. We thus obtain

 \[\beta'=\phi'_{X_{n+1},X_{0}}(\overline{\Id_{X_{0}}})=(\phi_{X_{n+1},X_{0}}\circ\overline{\C}(-,s))(\overline{\Id_{X_{0}}})=\phi_{X_{n+1},X_{0}}(\overline s).\]
Consider the following commutative diagram
 \[\xymatrix@C+5em{
    \overline{\C}(\tau_n X_{n+1},\tau_n X_{n+1})\ar[r]^-{\phi_{{X_{n+1}},\tau_n X_{n+1}}}\ar[d]_-{\overline{\C}(s,\tau_n X_{n+1})}
      &D\E({ X_{n+1}},\tau_n X_{n+1})\ar[d]^-{D\E({X_{n+1}},s)}\\
    \overline{\C}(X_{0},\tau_n X_{n+1})\ar[r]^-{\phi_{{ X_{n+1}},X_{0}}}
      &D\E({ X_{n+1}},X_{0}).
  }\]
Since $\phi_{{ X_{n+1}},\tau_n X_{n+1}}(\overline{\Id_{\tau_n X_{n+1}}})=\beta$ and $\overline{\C}(s,\tau_n X_{n+1})(\overline{\Id_{\tau_n X_{n+1}}})=\overline s$, we have
\[\beta' = D\E(X_{n+1},s)(\beta) = \beta\circ\E(X_{n+1},s)=\alpha\circ\E(\vartheta_{X_{n+1}},\tau_n X_{n+1})\circ\E(X_{n+1},s).\]
Thus we have
\[0\not = \beta'(\eta) = \alpha(\vartheta_{X_{n+1}}^\ast(s_\ast\eta)) = \alpha(s_\ast(\vartheta_{X_{n+1}}^\ast\eta)),\]
which implies that the distinguished $n$-exangle
$$U_{\bullet}:X_0\xrightarrow{}U_1\xrightarrow{}U_2\xrightarrow{}\cdots\xrightarrow{}U_n\xrightarrow{}\tau_n^-\tau_n X_{n+1}\overset{\vartheta_{X_{n+1}}^\ast\eta}{\dashrightarrow}$$
is non-split. We claim that $\vartheta_{X_{n+1}}\colon\tau_n^-\tau_n X_{n+1}\to X_{n+1}$ is a split epimorphism in $\C$. Otherwise, since $\eta$ is almost split,
we have the following commutative diagram
$$\xymatrix{
U_{\bullet}:& X_0\ar[r]^{}\ar@{}[dr] \ar@{=}[d]^{} &U_1 \ar[r]^{} \ar@{}[dr]\ar[d]^{}&\cdot\cdot\cdot \ar[r]^{} \ar@{}[dr]&U_{n-1} \ar[r]^{}\ar@{}[dr]\ar[d]^{} &U_n \ar[r]^{} \ar@{}[dr]\ar[d]^{}&\tau_n^-\tau_n X_{n+1} \ar@{}[dr]\ar[d]^{\vartheta_{X_{n+1}}}\ar@{-->}[dl]^{} \ar[r]^-{\vartheta_{X_{n+1}}^\ast\eta} &\\
X_{\bullet}:& {X_0}\ar[r]^{} &{X_1}\ar[r]^{}&\cdot\cdot\cdot\ar[r]^{} &{X_{n-1}}  \ar[r]^{} &{X _n}\ar[r]^{}  &{X_{n+1}} \ar@{-->}[r]^-{\eta} &.}
$$
By Lemma \ref{y1}, the top distinguished $n$-exangle is split, which is a contradiction. Moreover, $\tau_n^-\tau_n X_{n+1} \cong X_{n+1}$ in $\underline\C$ by Proposition \ref{lem11}.
Thus $\underline{\vartheta_{X_{n+1}}}$ is an isomorphism in $\underline{\C_r}$.
\end{proof}

\section{Auslander-Reiten ${n}$-exangles via morphisms determined by objects}
\begin{definition}(Auslander \cite{Au})
Let $f\in \C(X, Y)$ and $C\in\C$. The morphism $f$ is called \emph{right $C$-determined} and $C$ is called a \emph{right determiner} of $f$, if the following condition is satisfied: each
$g\in \C(L, Y)$ factors through $f$, provided that for each $h\in\C(C, L)$ the morphism $g\circ h$ factors through $f$.

If moreover $C$ is a direct summand of any right determiner of $f$, we call $C$ a \emph{minimal right determiner} of $f$.
\end{definition}

\begin{lemma}\label{le40}
Consider a morphism of distinguished $n$-exangles of the form
$$\xymatrix{
 Y_0\ar[r]^{\alpha_0}\ar@{}[dr] \ar@{=}[d]^{} &X_1 \ar[r]^{\alpha_1} \ar@{}[dr]\ar[d]^{h_1}&\cdot\cdot\cdot \ar[r]^{\alpha_{n-2}} \ar@{}[dr]&X_{n-1} \ar[r]^{\alpha_{n-1}}\ar@{}[dr]\ar[d]^{h_{n-1}} &X_n \ar[r]^{\alpha_n} \ar@{}[dr]\ar[d]^{h_n}&X_{n+1} \ar@{}[dr]\ar[d]^{h_{n+1}} \ar@{-->}[r]^-{\delta'} &\\
{Y_0}\ar[r]^{\beta_0} &{Y_1}\ar[r]^{\beta_1}&\cdot\cdot\cdot\ar[r]^{\beta_{n-2}} &{Y_{n-1}}  \ar[r]^{\beta_{n-1}} &{Y _n}\ar[r]^{\beta_n}  &{Y_{n+1}} \ar@{-->}[r]^-{\delta} &.}
$$
If $\beta_n$ is right $C$-determined for some object $C$ in $\C$, then $\alpha_n$ is also right $C$-determined.
\end{lemma}

\begin{proof}
Let $g\colon L\to X_{n+1}$ be a morphism such that for each $f\colon C\to L$, there exists some morphism $s\colon C\to X_n$
satisfying $g\circ f=\alpha_n\circ s$. Then we have $h_{n+1}\circ g\circ f=h_{n+1}\circ\alpha_n\circ s=\beta_n\circ h_{n}\circ s$.
Since $\beta_n$ is right $C$-determined, there exists some morphism $v\colon L\to Y _n$ such that $h_{n+1}\circ g=\beta_n\circ v$.
That is to say, we have the following commutative diagram

$$\xymatrix{ & && C\ar[r]^{f}\ar@{-->}[dr]^{s} &L\ar[r]^{g}\ar@/^1pc/@{-->}[dd]^(0.25){v} &X_{n+1}\ar@{}[dr]\ar@{=}[d]^{} &\\
 Y_0\ar[r]^{\alpha_0}\ar@{}[dr] \ar@{=}[d]^{} &X_1 \ar[r]^{\alpha_1} \ar@{}[dr]\ar[d]_{h_1}&\cdot\cdot\cdot \ar[r]^{\alpha_{n-2}} \ar@{}[dr]&X_{n-1} \ar[r]^{\alpha_{n-1}}\ar@{}[dr]\ar[d]_{h_{n-1}} &X_n \ar[r]^{\alpha_n} \ar@{}[dr]\ar[d]_{h_n}&X_{n+1} \ar@{}[dr]\ar[d]^{h_{n+1}} \ar@{-->}[r]^-{\delta'} &\\
{Y_0}\ar[r]^{\beta_0} &{Y_1}\ar[r]^{\beta_1}&\cdot\cdot\cdot\ar[r]^{\beta_{n-2}} &{Y_{n-1}}  \ar[r]^{\beta_{n-1}} &{Y _n}\ar[r]^{\beta_n}  &{Y_{n+1}} \ar@{-->}[r]^-{\delta} &.}
$$
Thus, in order to show that $\alpha_n$ is right $C$-determined, it suffices to find a morphism $r\colon L\to X_n$ such that $g=\alpha_n\circ r$.
Indeed, by  Definition \ref{def1}, we have the following commutative diagram with exact rows
\[\xymatrix@+1.5em{
 \C(L,X_n )\ar[r]^-{\C(T,\alpha_n )}\ar[d]^{\C(T,h_n )}
      & \C(L,X_{n+1})\ar[r]^-{\del'\ssh}\ar[d]^{\C(L,h_{n+1} )}
      &\E(L,Y_0)\ar@{=}[d]^{}\\
    \C(L,Y_n )\ar[r]^-{\C(T,\beta_n )}
      &\C(L,Y_{n+1} )\ar[r]^-{\del\ssh}
      &\E(L,Y_0).
  }\]
Since $\C(L,h_{n+1})(g)=h_{n+1}\circ g=\beta_n\circ v=\C(T,\beta_n)(v)\in\Im~{\C(T,\beta_n)}=\Ker~{\del\ssh}$, we have ${\del'\ssh}(g)={\del\ssh}(\C(L,h_{n+1})(g))=0$, and hence $h\in { \Ker}~\del'\ssh={\Im}~\C(T,\alpha_n)$. Thus there exists a morphism $r\colon L\to X_n$ such that $g=\alpha_n\circ r$.

\end{proof}

The proofs of the following two Lemmas are straightforward, we omit it.
\begin{lemma}\label{lem42}
Let $X_{\bullet}:X_0\xrightarrow{}X_1\xrightarrow{}X_2\xrightarrow{}\cdots\xrightarrow{}X_n\xrightarrow{\alpha}X_{n+1}\overset{\eta}{\dashrightarrow}$ be a distinguished $n$-exangle. Then a morphism $f\colon L\to X_{n+1}$ factors through $\alpha$ in $\C$
if and only if $\underline f$ factors through $\underline\alpha$ in $\underline\C$.
\end{lemma}

\begin{lemma}\label{lem43}
Let $C$ be an object in $\C$ and $$X_{\bullet}:X_0\xrightarrow{}X_1\xrightarrow{}X_2\xrightarrow{}\cdots\xrightarrow{}X_n\xrightarrow{\alpha}X_{n+1}\overset{\eta}{\dashrightarrow}$$ a distinguished $n$-exangle. Then the following statements are equivalent.
\begin{enumerate}
\item[$(1)$] $\alpha$ is right $C$-determined in $\C$.
\item[$(2)$] $\underline\alpha$ is right $C$-determined in $\underline\C$.
\end{enumerate}
\end{lemma}

As a consequence of Lemma \ref{lem43}, we have the following.

\begin{corollary}\label{cor}
Let $C,C'\in\C$ such that $C\simeq C'$ in $\underline\C$ and $$X_{\bullet}:X_0\xrightarrow{}X_1\xrightarrow{}X_2\xrightarrow{}\cdots\xrightarrow{}X_n\xrightarrow{\alpha}X_{n+1}\overset{\eta}{\dashrightarrow}$$ a distinguished $n$-exangle. Then a deflation $\alpha$
is right $C$-determined if and only if it is right $C'$-determined.
\end{corollary}
\begin{proof} The assertion follows by applying Lemma \ref {lem43} twice.
\end{proof}

\begin{lemma}\label{lem49}
Let $$X_{\bullet}:X_0\xrightarrow{}X_1\xrightarrow{}X_2\xrightarrow{}\cdots\xrightarrow{}X_n\xrightarrow{\alpha}X_{n+1}\overset{\eta}{\dashrightarrow}$$
be a distinguished $n$-exangle with $X_0\in\C_l$. Then $\alpha$ is right $\tau_n^- X_0$-determined.
\end{lemma}
\begin{proof}
Let $f\in\C(L,X_{n+1})$ such that for each $g\in\C(\tau_n^- X_0,L)$, the morphism $f \circ g$ factors through $\alpha$, that is,
there exists the following commutative diagram
$$
\xymatrix@C=0.5cm@R=0.3cm{&&&&&\tau_n^- X_0\ar[d]^{^{ g}}\ar@{..>}[ddl]^{ }&\\
&&&&&L\ar[d]^{^{f}}&\\
{X_0}\ar[r]^{} &{X_1}\ar[r]^{}&\cdot\cdot\cdot\ar[r]^{} &{X_{n-1}}  \ar[r]^{} &{X _n}\ar[r]^{\alpha}  &{X_{n+1}} \ar@{-->}[r]^-{\eta} &.}
$$
By Lemma \ref{y1}, $(f \circ g)^\ast\eta=0$. Since $X_0\in\C_l$, there exists a natural isomorphism
$$\psi_{-,X_0}:\underline{\C}(\tau_n^- X_0,-)\rightarrow D\E(-,X_0).$$
Set $\varphi: = \psi_{\tau_n^- X_0,X_0}(\underline{\Id_{\tau_n^- X_0}})$.
By the naturality of $\psi_{-,X_0}$, we have the following commutative diagram
\[\xymatrix@C=1.6cm{
\underline{\C}(\tau_n^- X_0,\tau_n^- X_0)\ar[r]^-{\psi_{\tau_n^- X_0,X_0}}\ar[d]_-{\underline{\C}(\tau_n^- X_0,\underline{g})}
&D\E(\tau_n^- X_0,X_0)\ar[d]^-{D\E(g,X_0)}\\
\underline{\C}(\tau_n^- X_0,L)\ar[r]^-{\psi_{L,X_0}}
&D\E(L,X_0).
}\]
Thus, we have $
\psi_{L,X_0}(\underline g)
= D\E(g,X_0)(\varphi)
= \varphi\circ\E(g,X_0).
$
Therefore,
\[
\psi_{L,X_0}(\underline g)(f^\ast\eta)
= \varphi(g^\ast f^\ast\eta)
= \varphi((f\circ g)^\ast\eta)
= 0.
\]
Note that $\psi_{L,X_0}(\underline g)$ runs over all maps in $D\E(L,X_0)$, when $\underline g$ runs over all morphisms in $\underline\C(\tau_n^- X_0,L)$.
It follows that $f^\ast\eta=0$ and hence, by Lemma \ref{y1}, the morphism $f$ factors through $\alpha$. Thus $\alpha$ is right $\tau_n^- X_0$-determined.
\end{proof}
The following result is essentially contained in \cite[Theorem 4.3]{JL}. However, it can be
extended to our setting.

\begin{proposition}\label{pro1}
\rm Let $C\in\C_r$ and $X_{n+1}\in\C$, and let $H$ be a right $\End_{\C}(C)$-submodule of $\C(C,X_{n+1})$ satisfying $\P(C,X_{n+1})\subseteq H$.
Then there exists a distinguished $n$-exangle
$$X_{\bullet}:X_0\xrightarrow{}X_1\xrightarrow{}X_2\xrightarrow{}\cdots\xrightarrow{}X_n\xrightarrow{\alpha}X_{n+1}\overset{\eta}{\dashrightarrow}$$
such that $\alpha$ is right $C$-determined, $X_0\in{\rm{\add}}~(\tau_n C)$ and $H={\rm{\Im}}~\C(C,\alpha)$.
\end{proposition}
\begin{proof}
By Proposition \ref{lem11}, we have $\tau_n C\in\C_l$ and $\tau_n^-\tau_n C\cong C$ in $\underline\C$. Then there exists a natural isomorphism
\[
\phi_{-,\tau_nC}\colon\underline\C(C,-)
\longrightarrow D\E(-,\tau_n C).
\]
Set $\varphi:=\phi_{C,\tau_nC}(\underline{\Id_C})$. By the naturality of $\phi_{-,\tau_nC}$, for each object $Y_{n+1}$ and each $f\in\C(C,Y_{n+1})$,
we have the following commutative diagram
\[\xymatrix@C=1.6cm{
\underline{\C}(C,C)\ar[r]^-{\phi_{C,\tau_nC}}\ar[d]_-{\underline{\C}(C,\underline{f})}
&D\E(C,\tau_nC)\ar[d]^-{D\E(f,\tau_nC)}\\
\underline{\C}(C,Y_{n+1})\ar[r]^-{\phi_{Y_{n+1},\tau_nC}}
&D\E(Y_{n+1},\tau_nC).
}\]
Hence $
\phi_{Y_{n+1},\tau_nC}(\underline f)
= D\E(f,\tau_nC)(\varphi)
= \varphi\circ\E(f,\tau_nC).
$ Then for each $\mu\in\E(Y_{n+1},\tau_nC)$, we have
\[
\phi_{Y_{n+1},\tau_nC}(\underline f)(\mu)
= \varphi(\E(f,\tau_nC)(\mu))
= \varphi(f^\ast\mu).
\]
Set $\underline H:=H/\P(C,X_{n+1})$ and
\[
\underline H^\perp:
= \set{\mu\in\E(X_{n+1},\tau_nC) \middle|
\phi_{X_{n+1},\tau_nC}(\underline h)(\mu) = 0
\mbox{ for each }
\underline h \in \underline H}.
\]
We observe that $\underline H^\perp$ is a left $\End_{\C}(C)$-submodule of $\E(X_{n+1},\tau_nC)$.
where the left $\End_{\C}(C)$-module structure is given by $f\mu=:\tau_n(\underline{f})_\ast\mu$ for any $f\in\End_{\C}(C)$.
Then there exists finitely many $\eta_1,\eta_2,\cdots,\eta_n$ in $\E(X_{n+1},\tau_nC)$ such that $\underline H^\perp = \sum\limits_{i=1}^n \eta_i\End_{\C}(C)$.
Assume that $\eta_i$ is represented by the distinguished $n$-exangle
$$\tau_nC\xrightarrow{}X_1^i\xrightarrow{}X_2^i\xrightarrow{}\cdots\xrightarrow{}X_n^i\xrightarrow{\alpha_i}X_{n+1}\overset{\eta_i}{\dashrightarrow}$$
for each $i=1,2,\cdots,n$. Then $\alpha_i$ is right $\tau_n^-\tau_n C$-determined by Lemma~\ref{lem49}, and thus $\alpha_i$ is right $C$-determined by Corollary~\ref{cor}.
Note that $\bigoplus\limits_{i=1}^n \alpha_i$ is a deflation. One can check that $\bigoplus\limits_{i=1}^n \alpha_i$ is right $C$-determined.
Consider the following commutative diagram
$$\xymatrix{
\bigoplus\limits_{i=1}^n\tau_nC\ar[r]^{}\ar@{}[dr] \ar@{=}[d]^{} &X_1 \ar[r]^{} \ar@{}[dr]\ar[d]^{}&\cdot\cdot\cdot \ar[r]^{}\ar@{}[dr]&X_{n-1} \ar[r]^{}\ar@{}[dr]\ar[d]^{} &X_n \ar[r]^{\alpha} \ar@{}[dr]\ar[d]^{}&X_{n+1} \ar@{}[dr]\ar[d]^{\Delta} \ar@{-->}[r]^-{\vartriangle^\ast({\bigoplus\eta_i})} &\\
{\bigoplus\limits_{i=1}^n\tau_nC}\ar[r]^{} &{\bigoplus\limits_{i=1}^nX_1^i}\ar[r]^{}&\cdot\cdot\cdot\ar[r]^{} &{\bigoplus_{i=1}^nX_{n-1} ^i}  \ar[r]^{} &{\bigoplus\limits_{i=1}^nX_{n} ^i}\ar[r]^{\bigoplus\limits_{i=1}^n\alpha_{i}}  &{\bigoplus\limits_{i=1}^nX_{n+1} } \ar@{-->}[r]^-{\bigoplus\eta_i} &,}
$$
where $\Delta=(\mbox{Id}_{X_{n+1}},\mbox{Id}_{X_{n+1}},...,\mbox{Id}_{X_{n+1}})^{T}$. We have that $\alpha$ is a deflation and $\bigoplus\limits_{i=1}^n \tau_n C\in\add(\tau_n C)$. By Lemma~\ref{le40},
we have that $\alpha$ is right $C$-determined. By a direct verification, we have
$
\Im ~\C(C,\alpha)
= \bigcap\limits_{i=1}^n \Im~\C(C,\alpha_i).
$
For each $i=1,2,\cdots,n$, set
\[
^\perp(\eta_i \End_{\C}(C)):
=\set{\underline h \in \underline\C(C,X_{n+1}) \middle|
\phi_{X_{n+1},\tau_nC}(\underline{h})(\mu) = 0
\mbox{ for each }
\mu \in \eta_i \End_{\C}(C)}.
\]
We observe that $^\perp(\eta_i\End_{\C}(C))$ is a right $\End_{\C}(C)$-submodule of $\underline\C(C,X_{n+1})$.
Note that $\P(C,X_{n+1})\subseteq\Im~\C(C,\alpha_i)$, since $\alpha_i$ is a deflation. We claim that $$^\perp(\eta_i\End_{\C}(C))= \Im ~\C(C,\alpha_i)/\P(C,X_{n+1}).$$
Let $h\colon C\to X_{n+1}$ be a morphism in $\Im~\C(C,\alpha_i)$. We have $h^\ast\eta_i$ splits. Then
Then we have
\[
\phi_{X_{n+1},\tau_nC}(\underline h)(\tau_n(\underline f).\eta_i)
= \varphi((\tau_n(\underline f).\eta_i).h)
= \varphi(\tau_n(\underline f) (\eta_i.h))
= 0,
\]
for each $f\colon C\to C$. It follows that
$\Im~\C(C,\alpha_i)/\P(C,X_{n+1})\subseteq {^\perp(\eta_i\End_{\C}(C))}$.

On the other hand, let $h\in\C(C,X_{n+1})$  such that $\underline h\in{^\perp(\eta_i\End_{\C}(C))}$.
Then we have $\phi_{X_{n+1},\tau_nC}(\underline h)(\tau_n(\underline f)_\ast\eta_i)=0$ for each $f\colon C\to C$.
Consider the following commutative diagram
\[\xymatrix@C=1.6cm{
\underline{\C}(C,X_{n+1})\ar[r]^-{\phi_{X_{n+1},\tau_nC}}\ar[d]_-{\underline{\C}(\underline{f},X_{n+1})}
&D\E(X_{n+1},\tau_nC)\ar[d]^-{D\E(X_{n+1},\tau_n\underline{f})}\\
\underline{\C}(C,X_{n+1})\ar[r]^-{\phi_{X_{n+1}},\tau_nC}
&D\E(X_{n+1},\tau_nC).
}\]
By a diagram chasing, we have
\[\begin{split}
\phi_{X_{n+1},\tau_nC}(\underline{h} \circ \underline{f})
&= \phi_{X_{n+1},\tau_nC} ( {\underline \C}(\underline{f}, X_{n+1}) (\underline{h}) )\\
&= (D\E(X_{n+1}, \tau_n(\underline{f})) \circ \phi_{X_{n+1},\tau_nC}) (\underline{h})\\
&= \phi_{X_{n+1},\tau_nC}(\underline{h}) \circ \E(X_{n+1}, \tau_n(\underline{f})).\\
\end{split}\]
Then for each $\eta_i$, we have $
\phi_{X_{n+1},\tau_nC}(\underline h\circ\underline f)(\eta_i)
= \phi_{X_{n+1},\tau_nC}(\underline h)(\tau_n(\underline f)_\ast\eta_i)
= 0.
$
This implies
\[
\phi_{C,\tau_nC}(\underline f)(h^\ast\eta_i)
= \varphi(f^\ast h^\ast\eta_i)
= \varphi((h\circ f)^\ast\eta_i)
= \phi_{X_{n+1},\tau_nC}(\underline{h\circ f})(\eta_i)
= 0.
\]
We observe that $\phi_{C,\tau_nC}(\underline f)$ runs over all maps in $D\E(C,\tau_n C)$, when $\underline f$ runs over all morphisms
in $\underline\End_\C(C)$. It follows that $h^\ast\eta_i$ splits and the morphism $h$ factors through $\alpha_i$.
We then obtain $h\in \Im~\C(C,\alpha_i)$, which means that
${^\perp(\eta_i\End_{\C}(C))}\subseteq\Im~\C(C,\alpha_i)/\P(C,X_{n+1})$. Moreover, we observe that
\[
\underline H
= {^\perp(\underline H^\perp)}
= {^\perp(\sum_{i=1}^n\eta_i\End_{\C}(C))}
= \bigcap_{i=1}^n {^\perp(\eta_i\End_{\C}(C))},
\]
where the first equality follows from the isomorphism $\phi_{X_{n+1},\tau_nC}$. It follows that
\[
\underline H
= \bigcap_{i=1}^n \Im~ \C(C,\alpha_i)/\P(C,X_{n+1})
= \Im~ \C(C,\alpha)/\P(C,X_{n+1}).
  \]
Then the assertion follows since $\P(C,X_{n+1})\subseteq H$.
\end{proof}
Before consider the converse of Proposition \ref{pro1}, we need some preparations.
In what follows, we always assume that the following condition, analogous to the (WIC) Condition
in \cite[Condition 5.8]{NP}.
\begin{condition}\label{cd}
Let $f \in \C(A,B)$, $g \in\C(B,C)$ be any composable pair of morphisms.  Consider the following
conditions.

(1) If $g \circ f$ is a deflation, then so is $g$.

(2) If $g \circ f$ is an inflation, then so is $f$.
\end{condition}
A morphism $g\colon Z\to Y$ is said to \emph{almost factor through} $f\colon X\to Y$, if $g$ does not factor through $f$,
and for each object $T$ and each morphism $h\in\rad_{\C}(T,Z)$, the morphism $g \circ h$ factors through $f$ (see \cite{Ri}).

\begin{lemma}\label{ras}
Consider a morphism of distinguished $n$-exangles as follows:
$$\xymatrix{
 Y_0\ar[r]^{\alpha_0}\ar@{}[dr] \ar@{=}[d]^{} &X_1 \ar[r]^{\alpha_1} \ar@{}[dr]\ar[d]^{h_1}&\cdot\cdot\cdot \ar[r]^{\alpha_{n-2}} \ar@{}[dr]&X_{n-1} \ar[r]^{\alpha_{n-1}}\ar@{}[dr]\ar[d]^{h_{n-1}} &X_n \ar[r]^{\alpha_n} \ar@{}[dr]\ar[d]^{h_n}&X_{n+1} \ar@{}[dr]\ar[d]^{h_{n+1}} \ar@{-->}[r]^-{\delta'} &\\
{Y_0}\ar[r]^{\beta_0} &{Y_1}\ar[r]^{\beta_1}&\cdot\cdot\cdot\ar[r]^{\beta_{n-2}} &{Y_{n-1}}  \ar[r]^{\beta_{n-1}} &{Y _n}\ar[r]^{\beta_n}  &{Y_{n+1}} \ar@{-->}[r]^-{\delta} &.}
$$
If $X_{n+1}$ is indecomposable and $h_{n+1}$ almost factors through $\beta_n$, then $\alpha_n$ is right almost split.
\end{lemma}

\begin{proof}
Since $h_{n+1}$ does not factors through $\beta_n$,  $\alpha_n$ is not a split epimorphism. Given an object $L$,
assume that $r\colon L\to X_{n+1}$ is not a split epimorphism.
Since $X_{n+1}$ is indecomposable, we conclude $r\in\rad_{\C}(L,X_{n+1})$. Thus there exists  $s\in\C(L, Y _n)$
such that $h_{n+1}\circ r = \beta_n\circ s$. By using an argument similar to that in the proof of Lemma \ref{le40},
there exists  $t\in\C(L,X_{n})$ such that $r=\alpha_n\circ t$. It follows that $\alpha_n$ is right almost split.
\end{proof}

We recall from \cite{Ch} that two morphisms $f\colon X\to Y$ and $f'\colon X'\to Y$ are called \emph{right equivalent}
if $f$ factors through $f'$ and $f'$ factors through $f$. Assume that $f$ and $f'$ are right equivalent.
Given an object $C$, we have that $f$ is right $C$-determined if and only if so is $f'$.

Under the Condition \ref{cd}, the following result is straightforward.
\begin{lemma}\label{lem47} Suppose that $f$ and $f'$ are right equivalent. Then $f$ is a deflation if and only if $f'$ is a deflation.
\end{lemma}
Given a morphism $f\colon X_{n}\to X_{n+1}$ , we call a right minimal morphism $f'\colon Z_{n}\to X_{n+1}$ the \emph{right minimal version} of $f$,
if $f$ and $f'$ are right equivalent. Assume that $f\colon X_{n}\to X_{n+1}$ is a deflation and there exists a distinguished $n$-exangle
$$Z_0\xrightarrow{}Z_1\xrightarrow{}Z_2\xrightarrow{}\cdots\xrightarrow{}Z_n\xrightarrow{f'}X_{n+1}\overset{}{\dashrightarrow}$$ then
we call $Z_0$ a \emph{intrinsic weak $n$-kernel} of $f$. Dually, we can define intrinsic weak $n$-cokernel of a inflation. Since $\C$ is Krull-Schmidt, each morphism $f\colon X\to Y$ has a right minimal version (see
\cite[Theorem~1]{Bi}).

\begin{lemma}\label{2st}
Let $f\colon X_{n}\to X_{n+1}$ be a right minimal deflation. Then there exists a distinguished $n$-exangle of the form
$$X_0\xrightarrow{f_{0}}X_1\xrightarrow{f_{1}}\cdots\xrightarrow{f_{n-2}}X_{n-1}\xrightarrow{f_{n-1}}X_n\xrightarrow{f}X_{n+1}\overset{\delta}{\dashrightarrow}$$ such that $f_{i}\in\rad_{\C}$ for every $0\leq i\leq n-1$.

\end{lemma}
\begin{proof} Let $X_0\xrightarrow{f_{0}}X_1\xrightarrow{f_{1}}\cdots\xrightarrow{f_{n-2}}X_{n-1}\xrightarrow{f_{n-1}}X_n\xrightarrow{f}X_{n+1}\overset{\delta}{\dashrightarrow}$ be a distinguished $n$-exangle such that the number of indecomposable direct summand of $\bigoplus\limits_{i=0}^{n-1}X_i$ is minimal. If $f_i\notin \rad_{\C}$ for some $0\leq i\leq n-2$, then we may write $X_i=U\oplus \tilde{X_i}$ and $X_{i+1}=U\oplus \tilde{X}_{i+1}$ such that $f_{i}=\left[
              \begin{smallmatrix}
               1&0\\
                0&\tilde{f_{i}}
              \end{smallmatrix}
            \right]$ for some $\tilde{f_{i}}:\tilde{X_i}\rightarrow\tilde{X}_{i+1}$, where $U$ is indecomposable. Then
            $$X_0\xrightarrow{}X_1\xrightarrow{}\cdots\xrightarrow{}\tilde{X_i}\xrightarrow{\tilde{f_{i}}}\tilde{X}_{i+1}\xrightarrow{}X_{i+2}\xrightarrow{}\cdots \xrightarrow{} X_{n-1}\xrightarrow{}X_n\xrightarrow{f}X_{n+1}\overset{\eta}{\dashrightarrow}$$ is a distinguished $n$-exangle by Proposition 3.3 in \cite{HLN}. This is a contradiction with the minimal of indecomposable direct summand in its terms. Moreover, since $f$ is right minimal, we have $f_{n-1}\in\rad_{\C}$ by Lemma 1.1 in \cite{JK}.

\end{proof}

\begin{lemma}\label{lst}
Let $a_n\colon X_{n}\to X_{n+1}$ be a deflation and $C$ an indecomposable object in $\C$. If there exists a morphism $h_{n+1}\colon C\to X_{n+1}$
which almost factors through $a_n$, then there exists an {Auslander-Reiten $n$-exangle}
$$Z_{\bullet}:Z_0\xrightarrow{\gamma_{0}}Z_1\xrightarrow{\gamma_{1}}Z_2\xrightarrow{}\cdots\xrightarrow{\gamma_{n-1}}Z_n\xrightarrow{\gamma_{n}}C\overset{\rho}{\dashrightarrow}$$ such that $Z_0$ is a direct summand of an intrinsic weak $n$-kernel of $a_n$.
\end{lemma}
\begin{proof}
Since $a_n\colon X_{n}\to X_{n+1}$ is a deflation, we have a distinguished $n$-exangle
$$X_{\bullet}:X_0\xrightarrow{\alpha_{0}}X_1\xrightarrow{\alpha_{1}}X_2\xrightarrow{\alpha_{2}}\cdots\xrightarrow{\alpha_{n-1}}X_n\xrightarrow{\alpha_n}X_{n+1}\overset{\sigma}{\dashrightarrow}.$$
Consider a morphism of distinguished $n$-exangles
$$\xymatrix{
 X_0\ar[r]^{\beta_0}\ar@{}[dr] \ar@{=}[d]^{} &Y_1 \ar[r]^{\beta_1} \ar@{}[dr]\ar[d]^{h_1}&\cdot\cdot\cdot \ar[r]^{\beta_{n-2}} \ar@{}[dr]&Y_{n-1} \ar[r]^{\beta_{n-1}}\ar@{}[dr]\ar[d]^{h_{n-1}} &Y_n \ar[r]^{\beta_n} \ar@{}[dr]\ar[d]^{h_n}&C \ar@{}[dr]\ar[d]^{h_{n+1}} \ar@{-->}[r]^-{\delta} &\\
{X_0}\ar[r]^{\alpha_0} &{X_1}\ar[r]^{\alpha_1}&\cdot\cdot\cdot\ar[r]^{\alpha_{n-2}} &{X_{n-1}}  \ar[r]^{\alpha_{n-1}} &{X _n}\ar[r]^{\alpha_n}  &{X_{n+1}} \ar@{-->}[r]^-{\sigma} &.}
$$
Then $\beta_n$ is a deflation. By Lemma~\ref{ras}, $\beta_n$ is right almost split.
Let $\gamma_n\colon Z_n\to C$ be the right minimal version of $\beta_n$. Then $\gamma_n$ is a right almost split deflation by Lemma~\ref{lem47}. Hence there exists
a distinguished $n$-exangle $$Z_{\bullet}:Z_0\xrightarrow{\gamma_{0}}Z_1\xrightarrow{\gamma_{1}}Z_2\xrightarrow{\gamma_{2}}\cdots\xrightarrow{\gamma_{n-1}}Z_n\xrightarrow{\gamma_n}C\overset{\rho}{\dashrightarrow},$$ where we may assume $\gamma_{i}\in\rad_{\C}$ for every $0\leq{i}\leq n-1$ by Lemma \ref{2st}. Since $\beta_n$ and $\gamma_n$ are right equivalent, we have the following commutative diagram of distinguished $n$-exangles by the dual of Lemma \ref {a2}
$$\xymatrix{Z_0\ar[r]^{\gamma_0}\ar@{}[dr] \ar@{-->}[d]^{f_0} &Z_1 \ar[r]^{\gamma_1} \ar@{}[dr]\ar@{-->}[d]^{f_1}&\cdot\cdot\cdot \ar[r]^{\gamma_{n-2}} \ar@{}[dr]&Z_{n-1} \ar[r]^{\gamma_{n-1}}\ar@{}[dr]\ar@{-->}[d]^{f_{n-1}} &Z_n \ar[r]^{\gamma_n} \ar@{}[dr]\ar[d]^{f_n}&C \ar@{}[dr]\ar@{=}[d]^{} \ar@{-->}[r]^-{\rho} &\\
 X_0\ar[r]^{\beta_0}\ar@{}[dr] \ar@{-->}[d]^{g_0} &Y_1 \ar[r]^{\beta_1} \ar@{}[dr]\ar@{-->}[d]^{g_1}&\cdot\cdot\cdot \ar[r]^{\beta_{n-2}} \ar@{}[dr]&Y_{n-1} \ar[r]^{\beta_{n-1}}\ar@{}[dr]\ar@{-->}[d]^{g_{n-1}} &Y_n \ar[r]^{\beta_n} \ar@{}[dr]\ar[d]^{g_n}&C \ar@{}[dr]\ar@{=}[d]^{} \ar@{-->}[r]^-{\delta} &\\
{Z_0}\ar[r]^{\gamma_0} &{Z_1}\ar[r]^{\gamma_1}&\cdot\cdot\cdot\ar[r]^{\gamma_{n-2}} &{Z_{n-1}}  \ar[r]^{\gamma_{n-1}} &{Z _n}\ar[r]^{\gamma_n}  &{C} \ar@{-->}[r]^-{\rho} &.}
$$
Since $\gamma_n$ is right minimal, $g_n\circ f_n$ is an isomorphism. In a similar way of the proof in \cite[Lemma 3.12]{F}, we know that $g_1\circ f_1,\cdot\cdot\cdot,g_{n-1}\circ f_{n-1}$ are all isomorphism. We claim that $g_0\circ f_0$ is also an isomorphism. In fact, we have the following commutative diagram with exact rows by Lemma \ref {a1}
$$\xymatrix@C=1.2cm{
\C(Z_2,-)\ar[r]^{\mathcal{C}(\gamma_1, -)}\ar[d]^{{\C}(g_2f_2,-)}_\cong &\C(Z_1,-)\ar[r]^{\C(\gamma_0, -)}\ar[d]^{{\C}(g_1f_1,-)}_\cong & \C(Z_0,-) \ar[r]^{~\del\ush~} \ar[d]^{\C(g_0f_0,-)}& \mathbb{E}(C,-)\ar[r]^{\mathbb{E}(\gamma_n,  -)}\ar@{=}[d]& \mathbb{E}(Z_n,-)\ar[d]^{\mathbb{E}(g_nf_n,-)}_\cong\\
\C(Z_2,-)\ar[r]^{\C(\gamma_1, -)} &\C(Z_1,-)\ar[r]^{\C(\gamma_0, -)}\ar[r] &  \C(Z_0,-) \ar[r]^{~\del\ush~}& \mathbb{E}(C,-)\ar[r]^{\mathbb{E}(\gamma_n,-)}& \mathbb{E}(Z_n,-)
.}$$
By the Five lemma, we have that $\C(g_0f_0,-)$ is an isomorphism, then $g_0f_0$ is an isomorphism by the Yoneda's lemma. Hence $Z_0$ is a direct summand of $X_0$.

Next, we claim that $Z_{\bullet}:Z_0\xrightarrow{\gamma_{0}}Z_1\xrightarrow{\gamma_{1}}Z_2\xrightarrow{\gamma_{2}}\cdots\xrightarrow{\gamma_{n-1}}Z_n\xrightarrow{\gamma_n}C\overset{\rho}{\dashrightarrow}$ is an {Auslander-Reiten $n$-exangle}. By Definition \ref{222}, we only need to show that $\gamma_{0}$ is left almost split. Let $k_0\colon Z_0\to W_0$ be a morphism in $\C$ that is not a split monomorphism. We have the following commutative diagram of distinguished $n$-exangles by \rm (EA2$\op$)
$$\xymatrix{
 Z_0\ar[r]^{\gamma_0}\ar@{}[dr] \ar[d]^{k_0} &Z_1 \ar[r]^{\gamma_1} \ar@{}[dr]\ar[d]^{k_1}&\cdot\cdot\cdot \ar[r]^{\gamma_{n-2}} \ar@{}[dr]&Z_{n-1} \ar[r]^{\gamma_{n-1}}\ar@{}[dr]\ar[d]^{k_{n-1}} &Z_n \ar[r]^{\gamma_n} \ar@{}[dr]\ar[d]^{k_n}&C \ar@{}[dr]\ar@{=}[d]^{}  \ar@{-->}[r]^-{\rho} &\\
{W_0}\ar[r]^{\xi_0} &{W_1}\ar[r]^{\xi_1}&\cdot\cdot\cdot\ar[r]^{\xi_{n-2}} &{W_{n-1}}  \ar[r]^{\xi_{n-1}} &{W _n}\ar[r]^{\xi_n}  &{C} \ar@{-->}[r]^-{{k_0}_\ast\rho} &.}
$$
Suppose that $k_0$ does not factor through $\gamma_0$. Then $\xi_n$ is not a split epimorphism by Lemma \ref {y1}. Since  $\gamma_n$ is a right almost split morphism, and hence there exists a morphism $s_n\colon W_n\to Z_n$ such that $\gamma_n\circ s_n=\xi_n$. We then obtain the following commutative diagram of distinguished $n$-exangles by the dual of Lemma \ref{a2}
$$\xymatrix{Z_0\ar[r]^{\gamma_0}\ar@{}[dr] \ar[d]^{k_0} &Z_1 \ar[r]^{\gamma_1} \ar@{}[dr]\ar[d]^{k_1}&\cdot\cdot\cdot \ar[r]^{\gamma_{n-2}} \ar@{}[dr]&Z_{n-1} \ar[r]^{\gamma_{n-1}}\ar@{}[dr]\ar[d]^{k_{n-1}} &Z_n \ar[r]^{\gamma_n} \ar@{}[dr]\ar[d]^{k_n}&C \ar@{}[dr]\ar@{=}[d]^{} \ar@{-->}[r]^-{\rho} &\\
 W_0\ar[r]^{\xi_0}\ar@{}[dr] \ar@{-->}[d]^{s_0} &W_1 \ar[r]^{\xi_1} \ar@{}[dr]\ar@{-->}[d]^{s_1}&\cdot\cdot\cdot \ar[r]^{\xi_{n-2}} \ar@{}[dr]&W_{n-1} \ar[r]^{\xi_{n-1}}\ar@{}[dr]\ar@{-->}[d]^{s_{n-1}} &W_n \ar[r]^{\xi_n} \ar@{}[dr]\ar[d]^{s_n}&C \ar@{}[dr]\ar@{=}[d]^{} \ar@{-->}[r]^-{{k_0}_\ast\rho} &\\
{Z_0}\ar[r]^{\gamma_0} &{Z_1}\ar[r]^{\gamma_1}&\cdot\cdot\cdot\ar[r]^{\gamma_{n-2}} &{Z_{n-1}}  \ar[r]^{\gamma_{n-1}} &{Z _n}\ar[r]^{\gamma_n}  &{C} \ar@{-->}[r]^-{\rho} &.}
$$
Since $\gamma_n$ is right minimal, the morphism $s_n\circ k_n$ is an isomorphism. In a similar way of the proof in \cite[Lemma 3.12]{F}, we know that $s_1\circ k_1,\cdot\cdot\cdot,s_{n-1}\circ k_{n-1}$ are all isomorphism. We claim that $s_0\circ k_0$ is also an isomorphism. In fact, we have the following commutative diagram with exact rows by Lemma \ref {a1}
$$\xymatrix@C=1.2cm{
\C(Z_2,-)\ar[r]^{\mathcal{C}(\gamma_1, -)}\ar[d]^{{\C}(s_2k_2,-)}_\cong &\C(Z_1,-)\ar[r]^{\C(\gamma_0, -)}\ar[d]^{{\C}(s_1k_1,-)}_\cong & \C(Z_0,-) \ar[r]^{~\del\ush~} \ar[d]^{\C(s_0k_0,-)}& \mathbb{E}(C,-)\ar[r]^{\mathbb{E}(\gamma_n,  -)}\ar@{=}[d]& \mathbb{E}(Z_n,-)\ar[d]^{\mathbb{E}(s_nk_n,-)}_\cong\\
\C(Z_2,-)\ar[r]^{\C(\gamma_1, -)} &\C(Z_1,-)\ar[r]^{\C(\gamma_0, -)}\ar[r] &  \C(Z_0,-) \ar[r]^{~\del\ush~}& \mathbb{E}(C,-)\ar[r]^{\mathbb{E}(\gamma_n,-)}& \mathbb{E}(Z_n,-)
.}$$
By the Five lemma, we have $\C(s_0k_0,-)$ is an isomorphism, then $s_0k_0$ is an isomorphism by the Yoneda's lemma. Thus $k_0$ is a split monomorphism, which is a contradiction with our assumption. Hence $k_0$ factor through $\gamma_0$. The claim is proved.

\end{proof}

\begin{corollary}\label{cort}
A minimal right determiner of a deflation has no non-zero projective direct summands and lies in $\C_r$.
\end{corollary}
\begin{proof}It is similar to the proof of Corollary 5.3 in \cite{JL}, we omit it.
\end{proof}
The following result is essentially contained in \cite[Theorem 5.4]{JL}. However, it can be
extended to our setting.
\begin{proposition}\label{pro0}
For any deflation $\alpha\in\C(X,Y)$, the following statements are equivalent.
\begin{enumerate}
\item[$(1)$] $\alpha$ is right $C$-determined for some object $C$.
\item[$(2)$] The intrinsic weak $n$-kernel of $\alpha$ lies in $\C_l$.
\end{enumerate}
\end{proposition}

\begin{proof}$(1)\Rightarrow (2)$ We may assume that $C$ is a minimal right determiner of $\alpha$. Then we have $C\in\C_r$ by Corollary \ref{cort}. Notice that $\Im~\C(C,\alpha)$ is a right $\End_{\C}(C)$-submodule of $\C(C,Y)$. Since $\alpha$ is a deflation, we have $\P(C,Y)\subseteq\Im~\C(C,\alpha)$. By Proposition \ref{pro1}, there is a distinguished $n$-exangle
$$X_0\xrightarrow{\beta_0}X_1\xrightarrow{\beta_1}X_2\xrightarrow{\beta_2}\cdots\xrightarrow{\beta_{n-1}}X_n\xrightarrow{\beta}Y\overset{\rho}{\dashrightarrow}$$
such that $\beta$ is right $C$-determined, $X_0\in\rm{\add}~(\tau_n C)$ and $H=\rm{\Im}~\C(C,\alpha)=\C(C,\beta)$. For any $f\in\C(C,X)$ and $g\in\C(C,X_n)$, we have $\alpha\circ f$ and $\beta\circ g$ factor through $\beta$ and $\alpha$, respectively. Since $\alpha$ and $\beta$ are right $C$-determined, we have $\alpha$ and $\beta$ factor through each other. This shows that  $\alpha$ and $\beta$ are right equivalent.

Suppose $\alpha':X'_n\rightarrow Y$ is a right minimal version of $\alpha$. Then $\alpha'$ and $\beta$ are right equivalent. Let $$X'_0\xrightarrow{\alpha_0}X'_1\xrightarrow{\alpha_1}X'_2\xrightarrow{\alpha_2}\cdots\xrightarrow{\alpha_{n-1}}X'_n\xrightarrow{\alpha'}Y\overset{\delta}{\dashrightarrow}$$
be a distinguished $n$-exangle, where we may assume $\alpha_{i}\in\rad_{\C}$ for every $0\leq{i}\leq n-1$ by Lemma \ref{2st}. Thus $X'_0$ is a intrinsic weak $n$-kernel of $\alpha'$. Since  $\alpha'$ and $\beta$ are right equivalent, there exist $\varphi_{n}:X'_n\rightarrow X_n$ and $\psi_{n}:X_n\rightarrow X'_n$ such that $\alpha'=\beta\varphi_{n}$ and $\beta=\alpha'\psi_{n}$. We then obtain the following commutative diagram of distinguished $n$-exangles by the dual of Lemma \ref{a2}
$$\xymatrix{X'_0\ar[r]^{\alpha_0}\ar@{}[dr] \ar@{-->}[d]^{\varphi_0} &X'_1 \ar[r]^{\alpha_1} \ar@{}[dr]\ar@{-->}[d]^{\varphi_1}&\cdot\cdot\cdot \ar[r]^{\alpha_{n-2}} \ar@{}[dr]&X'_{n-1} \ar[r]^{\alpha_{n-1}}\ar@{}[dr]\ar@{-->}[d]^{\varphi_{n-1}} &X'_n \ar[r]^{\alpha'} \ar@{}[dr]\ar[d]^{\varphi_n}&Y\ar@{}[dr]\ar@{=}[d]^{} \ar@{-->}[r]^-{\delta} &\\
 X_0\ar[r]^{\beta_0}\ar@{}[dr] \ar@{-->}[d]^{\psi_0} & X_1 \ar[r]^{\beta_1} \ar@{}[dr]\ar@{-->}[d]^{\psi_1}&\cdot\cdot\cdot \ar[r]^{\beta_{n-2}} \ar@{}[dr]& X_{n-1} \ar[r]^{\beta_{n-1}}\ar@{}[dr]\ar@{-->}[d]^{\psi_{n-1}} & X_n \ar[r]^{\beta} \ar@{}[dr]\ar[d]^{\psi_n}&Y \ar@{}[dr]\ar@{=}[d]^{} \ar@{-->}[r]^-{\rho} &\\
{X'_0}\ar[r]^{\alpha_0} &{X'_1}\ar[r]^{\alpha_1}&\cdot\cdot\cdot\ar[r]^{\alpha_{n-2}} &{X'_{n-1}}  \ar[r]^{\alpha_{n-1}} &{X'_n}\ar[r]^{\alpha'}  &{Y} \ar@{-->}[r]^-{\delta} &.}
$$
Using similar arguments as in the proof of Lemma \ref{lst}, we have $\psi_0\varphi_0$ is an isomorphism. Hence $X'_0$ is a direct summand of $X_0$. Then result follows since $X_0\in{\rm{\add}}~(\tau_n C)$ and $\tau_n C\in\C_l$.

$(2)\Rightarrow (1)$  Suppose $\alpha'\colon X'_n\to Y$ is a right minimal version of $\alpha$. Then $\alpha'$ is a deflation. Let
$$X'_0\xrightarrow{}X'_1\xrightarrow{}X'_2\xrightarrow{}\cdots\xrightarrow{}X'_n\xrightarrow{\alpha'}Y\overset{\delta}{\dashrightarrow}$$ be a distinguished $n$-exangle. Then $X'_0\in\C_l$ by the assumption of (2). By Lemma~\ref{lem49}, we have that $\alpha'$ is right $\tau_n^-(X'_0)$-determined. This shows that $\alpha$ is right $\tau_n^-(X'_0)$-determined
since $\alpha$ and $\alpha'$ are right equivalent.
\end{proof}

\begin{proposition}\label{prope}
Suppose $K$ is an object without non-zero injective direct summands, the following statements are equivalent.
\begin{enumerate}
\item[$(1)$] $K\in\C_l$.
\item[$(2)$] There exists some deflation $\alpha\colon X\to Y$, which is right $C$-determined for some object $C$ such that $K$ is the intrinsic $n$-kernel of $\alpha$.
\end{enumerate}
\end{proposition}

\begin{proof}
$(1)\Rightarrow (2)$ Let $K\in\C_l$. we may assume that $K=\bigoplus\limits_{i=1}^m K_i$, where $K_i$ is indecomposable, $i=1,2,\cdots,m$. Moreover, we have that each $K_i$ is non-injective by assumption. Hence, by Proposition \ref{desc}, for each $i=1,2,\cdots,m$,
there exists an {Auslander-Reiten $n$-exangle} starting at $K_i$
$$K_i\xrightarrow{}X^i_1\xrightarrow{}X^i_2\xrightarrow{}\cdots\xrightarrow{}X^i_n\xrightarrow{\alpha^i_n}X^i_{n+1}\overset{\rho_i}{\dashrightarrow}.$$
We have that $\bigoplus\limits_{i=1}^m \alpha^i_n$ is a deflation, then $\bigoplus\limits_{i=1}^m \alpha^i_n$ is right $\tau_n^-K$-determined by Lemma \ref{lem49}.  Moreover, we observe that $\alpha_i$ is right minimal, so $\bigoplus\limits_{i=1}^m \alpha^i_n$ is also right minimal.
It follows that $K$ is an intrinsic weak $n$-kernel of $\bigoplus\limits_{i=1}^m \alpha^i_n$.

$(2)\Rightarrow (1)$ It follows from Proposition \ref{pro0}.
\end{proof}

The following lemma is the converse of Proposition \ref{pro1}.

\begin{proposition}\label{hj}
\rm Let $C\in\C$. If for each $X_{n+1}\in\C$ and each right $\End_{\C}(C)$-submodule $H$ of $\C(C,X_{n+1})$ satisfying $\P(C,X_{n+1})\subseteq H$,
there exists a right $C$-determined deflation $\alpha_{n}\colon X_{n}\to X_{n+1}$ such that $\rm\Im~\C(C,\alpha)=H$,
then $C\in\C_r$.
\end{proposition}
\begin{proof}
It suffices to show that each non-projective indecomposable direct summand $C'$ of $C$ lies in $\C_r$ by Krull-Schmidt Theorem.
One can check that each morphism $f\in\P(C,C')$ is not a split epimorphism, since $C'$ is non-projective.
Notice that $\rad_\C(C,C')$ is formed by non split epimorphism, so $\P(C,C')\subseteq\rad_\C(C,C')$. Because $\rad_\C(C,C')$ is a right $\End_{\C}(C)$-submodule of $\C(C,C')$, by  assumption
there exists a right $C$-determined deflation $\alpha_{n}\colon X_{n}\to C'$ such that $\rad_\C(C,C')=\Im~\C(C,\alpha_{n})$.

Since $\Im~\C(C,\alpha_{n}) = \rad_\C(C,C')$ is a proper submodule of $\C(C,C')$, the deflation $\alpha_n$ is not a split epimorphism.
Thus $\Id_{C'}$ does not factor through $\alpha_n$. Let $r\in\rad_\C(L,C')$. For each $s\colon C\to L$,
the morphism $r\circ s$ lies in $\rad_\C(C,C')=\Im~\C(C,\alpha_n)$. It follows that $r\circ s$ factors through $\alpha_n$.
Since $\alpha_n$ is right $C$-determined, we have that $r$ factors through $\alpha_n$ and $\Id_{C'}$ almost factors through $\alpha_n$.
By Lemma~\ref{lst}, there exists an {Auslander-Reiten $n$-exangle} ending at $C'$.
Then the assertion follows from Proposition~\ref{desc}.
\end{proof}

We are now ready to state our main result.
\begin{theorem}\label{thm}
Let $\C$ be an $\Ext$-finite, Krull-Schmidt and $k$-linear $n$-exangulated category and $C\in\C$. Then the following statements are equivalent:
\begin{enumerate}
\item[$(1)$]
$C\in\C_r$.
\item[$(2)$]\rm
For each $X_{n+1}\in\C$ and each right $\End_{\C}(C)$-submodule $H$ of $\C(C,X_{n+1})$ satisfying $\P(C,X_{n+1})\subseteq H$, there exists a distinguished $n$-exangle
$$X_{\bullet}:X_0\xrightarrow{}X_1\xrightarrow{}X_2\xrightarrow{}\cdots\xrightarrow{}X_n\xrightarrow{\alpha}X_{n+1}\overset{\eta}{\dashrightarrow}$$
where $\alpha$ is right $C$-determined such that $H=\rm\Im~\C(C,\alpha)$.
\end{enumerate}
If moreover $C$ is non-projective indecomposable, then all statements $(1)-(6)$ are equivalent.
\begin{enumerate}
\item[$(3)$]\rm There exists an inflation $\alpha\colon X_{0}\to X_{1}$ whose intrinsic weak $n$-cokernel is $C$ such that $\alpha$ is left $K$-determined for some object $K$.

\item[$(4)$]\rm There exists an {Auslander-Reiten $n$-exangle} ending at $C$.
\item[$(5)$]\rm
There exists a non-split epimorphism deflation which is right $C$-determined.
\item[$(6)$]\rm
There exists a deflation $\alpha\colon X_{n}\to X_{n+1}$ and a morphism $f\colon C\to X_{n+1}$ such that $f$ almost factors through $\alpha$.

\end{enumerate}
\end{theorem}
\begin{proof}We summarize the proof as the following diagram
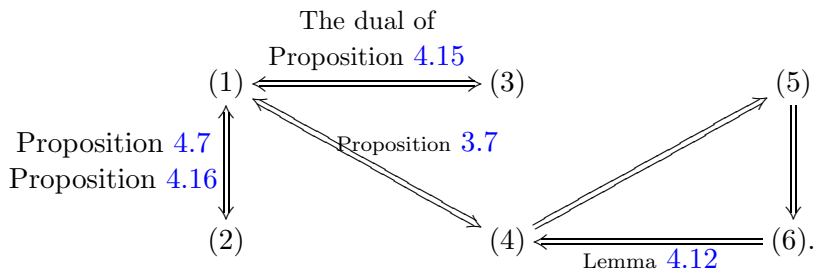
\[
\xymatrix@C=3cm@R=1.5cm{&({\rm 1})\ar@{<=>}[d]_{
\begin{matrix}
\text{Proposition}~\ref{pro1} \\
\text{Proposition}~\ref{hj}
\end{matrix}}\ar@{<=>}[rd]^{\text{Proposition}~\ref{desc}}&\ar@{<=>}[l]_{
\begin{matrix}
\text{\small The dual of} \\
\text{\small Proposition}~\ref{prope}
\end{matrix}}({\rm 3})&({\rm 5})\ar@{=>}[d]\ar@{<=}[ld]& \\
&({\rm 2})&({\rm 4})\ar@{<=}[r]_{\text{Lemma}~\ref{lst}}&({\rm 6}).&}
\]

Hence we only need to show that $(4)\Rightarrow (5)\Rightarrow (6)$.

It is easy to see that the right almost split deflation ending at $C$ is a non-split epimorphism and right $C$-determined. Then we have
$(4)\Rightarrow (5)$. Let $\alpha$ be a right $C$-determined deflation which is not a split epimorphism. We have that $C$ is a minimal
right determiner of $\alpha$. Thus $(5)\Rightarrow (6)$ holds true.
\end{proof}

\begin{remark}In Theorem \ref{thm}, when $\C$ is an exact category, it is just Theorem 6.3 in
\cite{JL}, when $\C$ is an extriangulated category, it is just Theorem 4.13 in
\cite{ZTH}, when $\C$ is an $n$-abelian category with enough projectives and enough injectives, it is just Theorem 4.16 (They only have (1), (2), (3) which are equivalent) in \cite{XLY}, and when $\C$ is an $(n+2)$-angulated category, it is a new phenomena.
\end{remark}


\textbf{Jian He}\\
Department of Mathematics, Nanjing University, 210093 Nanjing, Jiangsu, P. R. China\\
E-mail: \textsf{jianhe30@163.com}\\[0.3cm]
\textbf{Jing He}\\
College of Science, Hunan University of Technology and Business, 410205 Changsha, Hunan P. R. China\\
E-mail: \textsf{jinghe1003@163.com}\\[0.3cm]
\textbf{Panyue Zhou}\\
College of Mathematics, Hunan Institute of Science and Technology, 414006 Yueyang, Hunan, P. R. China.\\
E-mail: \textsf{panyuezhou@163.com}

\end{document}